\documentclass[preprint,12pt]{article} 
\usepackage{b}
\usepackage[a4paper]{geometry}
\usepackage[small,it]{caption}
\usepackage{empheq}

\usepackage{subcaption}
\usepackage{amsmath}
\usepackage{authblk}

\usepackage{cleveref}
\crefname{equation}{}{}
\Crefname{equation}{}{}
\crefname{eqnarray}{}{}
\Crefname{eqnarray}{}{}
\crefname{lemma}{lemma}{lemmas}
\crefname{lem}{lemma}{lemmas}
\crefname{cor}{corollary}{corollaries}
\crefname{thrm}{theorem}{theorems}

\def\bz{{\boldsymbol 0}}

\def\bv{{\boldsymbol v}}

\def\bh{{\boldsymbol h}}

\def\bt{{\boldsymbol t}}
\def\xx{{\boldsymbol x}}
\def\bx{{\boldsymbol x}}
\def\yy{{\boldsymbol y}}
\def\by{{\boldsymbol y}}

\def\uu{{\boldsymbol u}}
\def\bu{{\boldsymbol u}}
\def\be{{\boldsymbol e}}

\def\ff{{\boldsymbol f}}
\def\bg{{\boldsymbol g}}
\def\GG{{\bf G}}
\def\Gbh{G^{\textrm{BH}}}
\def\Glap{G^{\textrm{L}}}
\def\Ghelm{G^{\textrm{H}}}

\def\II{{\bf I}}
\def\TT{{\bf T}}
\def\bS{{\bf S}}
\def\bD{{\bf D}}
\def\bC{{\bf C}}
\def\bB{{\bf B}}
\def\cN{{\mathcal N}}

\def\cB{{\mathcal B}}
\def\cA{{\mathcal A}}
\def\cS{{\mathcal S}}
\def\cSk{{\mathcal S}_k}
\def\cD{{\mathcal D}}
\def\cDt{{\mathcal D}^{\intercal}}
\def\cDk{{\mathcal D}_{k}}
\def\cDkt{{\mathcal D}_{k}^{\intercal}}
\def\cW{{\mathcal W}}
\def\cJ{{\mathcal J}}

\def\cA{{\mathcal A}}
\def\cM{{\mathcal M}}

\def\sint{{\sin{(\theta)}}}
\def\cost{{\cos{(\theta)}}}

\def\jkr{J_{0}(kr)}
\def\jpkr{J_{0}'(kr)}
\def\jpk{J_{0}'(k)}
\def\jppkr{J_{0}''(kr)}
\def\jppk{J_{0}''(k)}

\def\bsigma{{\boldsymbol \sigma}}
\def\bpsi{{\boldsymbol \psi}}
\def\bmu{{\boldsymbol \mu}}
\def\bnu{{\boldsymbol \nu}}

\def\btau{{\boldsymbol \tau}}
\def\bigo{\mathcal{O}}
\def\littleo{o}

\def\Im{{\textrm{\textup{Im}}}}

\def\epsflam{{\epsilon_{\texttt{FLAM}}}}
\def\epscheb{{\epsilon_{\texttt{cheb}}}}

\def\cheb{{\texttt{cheb}}}

\newcommand{\cI}{\mathcal I}
\newcommand{\cC}{\mathcal C}
\newcommand{\bbmat}{\begin{bmatrix}}
\newcommand{\ebmat}{\end{bmatrix}}

\def\pv{\textrm{P.V.} }

\newtheorem{remark}{Remark}

\newtheorem{lem}{Lemma}
\newtheorem{cor}{Corollary}
\newtheorem{thrm}{Theorem}
\newtheorem{definition}{Definition}

\title{A boundary integral equation approach to
  computing eigenvalues of the Stokes operator}
\author{Travis Askham\thanks{
Department of Mathematical Sciences, New Jersey Institute of
Technology, USA.\\ 
email: travis.l.askham@njit.edu} ,\, 
Manas Rachh\thanks{Center for Computational Mathematics,
Flatiron Institute, USA. \\
email: mrachh@flatironinstitute.org}}

\date{}

\begin{document}

\maketitle
%%%%%%%%%%%%%%%%%%%%%%%%%%%%%%%%%%%%%%%%%%%%%%%%%%%%%%%%%%%%
\abstract{
  The eigenvalues and eigenfunctions of the Stokes 
  operator have been the subject of intense analytical
  investigation and have applications in the study
  and simulation of the Navier--Stokes equations.
  As the Stokes operator is a fourth-order
  operator,
  computing these eigenvalues and the corresponding
  eigenfunctions is a challenging task, particularly in
  complex geometries and at high frequencies.
  The boundary integral equation (BIE) framework
  provides robust and scalable eigenvalue computations
  due to (a) the reduction in the dimension of the problem to
  be discretized and (b)
  the absence of high frequency ``pollution'' when using
  a Green's function to represent propagating waves.
  In this paper, we detail the theoretical justification
  for a BIE approach to the Stokes eigenvalue problem on
  simply and multiply-connected planar domains, which entails
  a treatment of the uniqueness theory for oscillatory
  Stokes equations on exterior domains. 
  Then, using well-established techniques for
  discretizing BIEs,
  we present numerical results which confirm
  the analytical claims of the paper and demonstrate
  the efficiency of the overall approach.
}

% flatex input: [01intro.tex]
\section{Introduction}

The planar incompressible Stokes equations describe
creeping flows in two dimensions.
Let $\Omega \subset \R^2$
be a bounded domain with $C^2$ boundary denoted by $\Gamma$.
The Stokes eigenvalue problem is to find
values $k^2$ such that 

\begin{equation}
\begin{aligned}
  -\Delta \uu + \nabla p &= k^2 \uu \quad \textrm{in} \quad
  \Omega \label{eq:ostokes} \; , \\
  \nabla \cdot \uu &= 0 \; ,
\end{aligned}
\end{equation}
subject to boundary conditions, has a non-trivial solution $(\uu,p)$.
In this work, we consider the eigenvalue problem subject to 
the Dirichlet boundary condition,
\begin{equation}
  \uu = \bz \quad \textrm{on} \quad \Gamma \label{eq:ostokes_dir} \; .
\end{equation}
It is well known that the values $k^2$ are necessarily
real and positive and that there is a countable collection of such
values $0 < k_{1}^{2} \leq k_{2}^2 \leq \ldots \uparrow \infty$,
counting multiplicities.

\begin{remark}
  When $k = i\alpha$, the differential equation
  \cref{eq:ostokes} is known as the modified Stokes
  equation. As there appears to be no preferred
  name for the equation with real-valued $k$,
  we will refer to \cref{eq:ostokes} as the
  oscillatory Stokes equation.
\end{remark}

The eigenvalues (and eigenfunctions)
of the Stokes operator have applications in the
stability analysis of stationary solutions of the
Navier--Stokes equations \cite{osborn1976approximation},
in the study of decaying two dimensional turbulence
\cite{schneider2008final}, and as a trial basis for
numerical simulations of the Navier--Stokes
equations~\cite{batcho1994generalized}.
The eigenvalues and eigenfunctions of the Stokes
operator are also the subject of intense analytical
investigation
\cite{taylor1933buckling,szego1950membranes,
  polya1951isoperimetric,bramble1963pointwise,
  ashbaugh1996fundamental,leriche2004stokes,
  kelliher2009eigenvalues,antunes2011buckling},
especially as they relate to the eigenvalues and
eigenfunctions of the Laplacian.

Historically, the Stokes eigenvalue problem serves as a
common model problem for numerical eigenvalue analysis
with a fourth order operator (here, the bi-Laplacian).
Further, numerical simulation has long played an
important role in the analyses cited above --- both for
computing the eigenvalues and eigenfunctions
on domains of practical interest and in forming
new conjectures.

Borrowing the language of~\cite{zhao2015robust},
which concerns the eigenvalues of the Laplacian
(also known as the membrane or ``drum'' problem),
the numerical treatment of the Stokes eigenvalue
problem can be divided into two basic approaches.
The first class of methods
directly discretize the
differential operator, typically with a
finite element basis, and the eigenvalues are found
as the eigenvalues of the discrete system.
The second class of methods reformulate the 
oscillatory Stokes equations as a boundary integral
equation (BIE) which is discretized.
The eigenvalues are then found
by a nonlinear search for the values of
$k$ where the BIE is not invertible.

There is a large body of research on the first class of methods
for the Stokes eigenvalue problem.
We do not seek to review this literature here,
but point to \cite{johnson1974beam,
  rannacher1979nonconforming,
  mercier1981eigenvalue,bjorstad1999high,
  jia2009approximation,chen2006approximation,
  lovadina2009posteriori,huang2011numerical,
  carstensen2014guaranteed}
for some representative examples.

As noted in~\cite{zhao2015robust}, integral
equation based methods provide several advantages.
Because the BIE is defined on the
boundary alone, there is a reduction in the
dimension of the domain to be discretized.
This approach reduces the number of unknowns over finite
element discretizations, and does not suffer
from high-frequency ``pollution'' due to the 
large number of unknowns required to resolve
the computation when $k$ is large
\cite{babuska1997pollution} (which typically affects
finite element approaches to determining the eigenvalues).

Further, Zhao and Barnett~\cite{zhao2015robust}
show how to alleviate some of the costliness of the
nonlinear optimization introduced by formulating the 
problem as an integral equation.
The standard approach searches for ``V''-shaped minima
of the singular values of the BIE; see, for
instance, \cite{trefethen2006computed}.
Instead, Zhao and Barnett utilize the Fredholm
determinant (see \cref{sec:dets}) which, for certain
BIEs, is an analytic function of $k$ with roots
precisely when $k^2$ is an eigenvalue.
The Fredholm determinant can be estimated using
a Nystr\"{o}m discretization of the BIE
\cite{bornemann2010numerical,zhao2015robust}.
Then, the eigenvalues can be estimated efficiently
by using high order root finding methods applied
to the discretized determinant.

With the efficiency of the approach of
\cite{zhao2015robust} for the drum problem in mind,
we develop an integral equation based method for the
Stokes eigenvalue problem.
This requires that a layer
potential representation of the solution
of \cref{eq:ostokes} be given and that the resulting BIE
is not invertible precisely when $k^2$ is an eigenvalue.
The first requirement is straightforward to
satisfy because 
well-known layer potential representations for the
modified Stokes equation~\cite{Pozrikidis1992,biros2002embedded,
  jiang2013second,ladyzhenskaya1969mathematical}
are directly applicable.
Proving the invertibility of the associated operators
away from the eigenvalues is a more involved task
and forms the bulk of the theoretical component
of this paper.

\subsection{Relation to other work}

While integral equation based methods for the
related ``buckling'' eigenvalue problem
(which is equivalent  on simply connected domains
\cite{kelliher2009eigenvalues})
have been considered previously,
these typically relied on first-kind integral
equation formulations of the underlying PDE,
i.e. formulations in which the BIE operator is
compact \cite{kitahara2014boundary,antunes2011buckling}.
This is unsatisfying from a numerical
perspective, because the spectrum of a
compact operator either contains zero
or has zero as a limit point
(by design, the representations in
\cite{kitahara2014boundary,antunes2011buckling}
are not injective precisely when
$k^2$ is an eigenvalue).
This obscures the relation between
the non-invertibility of discrete approximations
of the operator and the eigenvalues;
in particular, common measures of the
``non-invertibility'' of a matrix, like the
smallest singular value or the determinant,
converge rapidly to zero for all values
of $k^2$ as the boundary is refined.
The measure of whether $k^2$ is an
approximate eigenvalue is then {\em relative to the
  current grid} for first kind formulations.

The classical single and double layer representations
for oscillatory Stokes considered in this paper
result in second kind equations, i.e. integral equations
of the form $\mathcal{I} - \mathcal{K}_k$ where
$\mathcal{K}_k$ is compact.
Such equations have a more satisfying theory
\cite{reed1972methods,colton1983integral,kress1989linear},
which translates well to numerical implementation
\cite{atkinson2009numerical,bornemann2010numerical,
  hackbusch2012integral,zhao2015robust}.
The use of a second kind representation is standard
for the drum problem \cite{backer2003numerical,zhao2015robust}
and was used recently to compute the vibrating
modes of thin, clamped plates~\cite{lindsay2018boundary}.

\subsection{Paper outline and contributions}

The rest of this paper proceeds as follows.
In \cref{sec:prelim}, we set the notation, provide some
mathematical preliminaries, and review 
properties of single and double layer potentials
for the oscillatory Stokes equations.
Then, in \cref{sec:analysis}, we develop the necessary
theory for proving the main results of this work 
(\cref{thm:dlmain,thm:cfmain}),
which show that the BIEs resulting from these
layer potential representations are not invertible
precisely when $k^2$ is an eigenvalue.
These theoretical developments include a detailed
discussion of the uniqueness of oscillatory Stokes
boundary value problems in exterior domains.
To the best of our knowledge, the invertibility
and uniqueness results are new to the literature.
\Cref{sec:dets} then outlines how the Fredholm determinant
can be used in the oscillatory Stokes context.
In \cref{sec:numerical}, we describe the numerical
methods we use to discretize the BIEs and to perform
determinant calculations.
While the underlying methods are well-established,
the combination of a high-order singular quadrature
rule and a fast-direct method for determinant evaluations
in a BIE framework is novel.
At the moderate frequencies considered in this
paper, we find that standard fast-direct
solvers provide a reasonably efficient determinant
evaluation.
We also present numerical experiments
which demonstrate some of the paper's analytical
claims as well as the effectiveness of the overall
framework.
Finally, we provide some concluding thoughts,
describe plans for future research,
and outline some open questions in
\cref{sec:conclusion}.

% flatex input end: [01intro.tex]

%
% flatex input: [02prelim.tex]
\section{Mathematical Preliminaries}
\label{sec:prelim}
In this paper, vector- and tensor-valued quantities
are denoted by bold letters (e.g. $\bh$ and $\mathbf{T}$). 
Subscript indices of non-bold characters (e.g. $h_j$ or $T_{ij\ell}$)
are used to denote the entries within a vector or tensor.
We use the standard Einstein summation convention; i.e., 
there is an implied sum taken over the repeated indices of 
any term (e.g. the symbol $a_{j} b_{j}$ is used to represent the sum
$\sum_{j} a_{j} b_{j}$).
If $\xx = (x_1,x_2)^\intercal$, then $\xx^\bot = (-x_2,x_1)^\intercal$.
Similarly, $\nabla^\bot = (-\partial_{x_2},\partial_{x_1})^\intercal$.
Upper-case script characters (e.g. $\mathcal{K}$) are reserved for
operators on Banach spaces, with $\mathcal{I}$ denoting the
identity. Given a set $X$, we denote the closure of $X$
by $\overline{X}$.

For a velocity field $\bu$ and pressure $p$, let $\bsigma(\bu,p)$
denote the Cauchy stress tensor is given by
\begin{equation}
\bsigma(\uu,p) = -p \II + 2 \be(\bu) \, ,
\end{equation}
where $\be(\bu)$ is the strain tensor given by
\begin{equation}
e_{ij}(\bu) = \frac{1}{2} \left( \partial_{x_j} u_i + \partial_{x_i} u_j \right) \; .
\end{equation}
When it is clear from context, we will drop the dependence of
$\bsigma$ on $\bu$ and $p$.
If $\Gamma$ is the boundary of a region $\Omega$ and $\bnu$ is the outward
normal to $\Gamma$, the surface traction $\bt$ on $\Gamma$ 
is the Neumann data, i.e. 
\begin{equation}
\bt = \bsigma \cdot \bnu \, .
\end{equation}

We seek solutions of \cref{eq:ostokes} in the space
\begin{equation}
  A(\Omega) = \{ (\bu,p) \textrm{ s.t. } \bu \in
  \left ( C^2(\Omega)\times C^2(\Omega) \right ) \cap
  \left ( C(\bar{\Omega}) \times C(\bar{\Omega}) \right ) \, , \,
  p \in C^1(\Omega) \cap C(\bar{\Omega})\} \; ,
\end{equation}
where $\Omega$ is an open domain.

\subsection{Green's functions}

Let $\mathcal{L}_x$ denote a linear differential operator. A fundamental
solution $G(\xx,\yy)$ of $\mathcal{L}_x$ satisfies the equation
$\mathcal{L}_x G(\xx,\yy) = \delta_y(\xx)$ in the distributional sense, i.e.
for sufficiently smooth $f$
\begin{equation}
  \mathcal{L}_x \int_{\R^2} G(\xx,\yy) f(\yy) \, d \by = f(\xx) \; .
  \nonumber
\end{equation}
We consider here
free-space Green's functions, i.e. fundamental solutions which satisfy appropriate
radiation conditions as $|\xx-\yy| \to \infty$.
The Green's function of the oscillatory biharmonic equation,
\begin{equation}
  \Delta ( \Delta + k^2 ) u = 0 \; , \label{eq:obiharm} \nonumber
\end{equation}
is given by 
\begin{equation}
  \Gbh(\xx,\yy) = \frac{1}{k^2}
  \left (\frac{1}{2\pi} \log |\xx-\yy| +
  \frac{i}{4} H_0^{(1)}(k|\xx-\yy|) \right ) \, ,
  \label{eq:Gbh}
\end{equation}
where $k$ is the Helmholtz parameter in the oscillatory biharmonic equation,
and $H_{0}^{1}(r)$ is the Hankel function of the first kind of order zero.
Note that this is a scaled difference of the Green's function for
Laplace, i.e.

\begin{equation}
  \Glap(\xx,\yy) = \frac{1}{2\pi} \log |\xx-\yy| \; , \nonumber
\end{equation}
and the Green's function for the Helmholtz equation

\begin{equation}
  \Ghelm(\xx,\yy) = -\frac{i}{4} H_0^{(1)}(k|\xx-\yy|) \; . \nonumber
\end{equation}

\subsection{The Fredholm Alternative}

We require some standard results from the theory of
Fredholm integral equations. Interested readers may
consult \cite{reed1972methods,colton1983integral,kress1989linear},
among others, for the relevant background.

We first recall some definitions.
Let $X$ and $Y$ be Banach spaces with a non-degenerate
bilinear form $\langle \cdot ,\cdot \rangle: X\times Y \to \C$.
\begin{itemize}
\item Two operators $\cA:X\to X$ and $\cB:Y\to Y$ are
adjoint operators if
$\langle A \phi,\psi \rangle = \langle \phi, B\psi \rangle$
for every $\phi \in X$ and $\psi \in Y$.
\item For an operator $\cM:X\to X$, we can define the range
  $R(\cM)$ as the set $\{\phi \in X: \exists \phi_0 \textrm{ with }
  \cM\phi_0 = \phi\}$ and the null space $N(\cM)$ as the
  set $\{\phi \in X: \cM \phi = 0 \}$.
\item An operator $\cA$ is said to be compact if
  $\overline{\cA V}$ is a compact set for any
  bounded subset $V\subset X$.
\item Given a subspace $V\subset X$,
  we can define the subspace $V^\perp\subset Y$ as the
  set $V^\perp = \{ \psi \in Y: \langle \phi,\psi \rangle = 0
  \textrm{ for each } \phi \in V \}$, with the analogous
  definition for subspaces of $Y$.
\end{itemize}

Operators of the form $\cI-\cA$
have existence and uniqueness properties analogous to
matrices. This is known as the Fredholm Alternative;
we present the version provided in \cite{colton1983integral}.

\begin{thrm}[Fredholm Alternative \cite{colton1983integral}]
  Let $X$ and $Y$ be Banach spaces and
  $\langle \cdot,\cdot \rangle: X\times Y \to \C$ be
  a bilinear form. Suppose that $\cA: X\to X$ and
  $\cB:Y \to Y$ are compact adjoint operators. Then
  $\dim N(\cI-\cA) = \dim N(\cI-\cB) \in \N$,
  $R(\cI-\cA) = N(\cI-\cB)^{\perp}$, and
  $R(\cI-\cB) = N(\cI-\cA)^\perp$.
\end{thrm}

\subsection{Properties of the oscillatory Stokes layer
  potentials}

Recall that, in the case $k=i\alpha$ for some real-valued $\alpha$,
the oscillatory Stokes equations \cref{eq:ostokes}
are known as the modified Stokes equations and are of particular
interest for their application to the analysis and numerical
simulation of unsteady flow
\cite{Pozrikidis1992,biros2002embedded,
  jiang2013second,ladyzhenskaya1969mathematical}.
The equations are well-studied in that setting and
integral representations which lead to second kind
integral equations have been developed. We review
some of the relevant results here, translating to
the oscillatory setting.

\subsubsection{Oscillatory Stokeslets and stresslets}
Consider the solution of
\cref{eq:ostokes} where a $\delta$-mass
centered at $\yy$ with strength $\ff$
has been added to the right-hand side of \cref{eq:ostokes}, i.e.

\begin{align}
  \nabla p - \Delta \uu - k^2 \uu &= \delta_\yy \ff \; ,
  \label{eq:ostokes_charge}  \\
  \nabla \cdot \uu &= 0 \; . \nonumber
\end{align}
Recall that

\begin{equation}
 \Delta \Glap(\xx,\yy) = \delta_\yy(\xx) \; . \label{eq:lapdelta}
\end{equation}
If we substitute \eqref{eq:lapdelta} into
\eqref{eq:ostokes_charge} and take the divergence,
we obtain

\begin{equation}
  p = \nabla \Glap(\xx,\yy) \cdot \ff \; . \nonumber
\end{equation}
We then have, formally,

\begin{align}
  \uu &= - (\Delta + k^2)^{-1} ( \Delta \Glap \ff
  - \nabla (\nabla \Glap \cdot \ff ) ) \nonumber \\
  &= \left ( -\Delta + \nabla \otimes \nabla \right )
  \Gbh \ff \; . \nonumber
\end{align}
The tensor

\begin{equation} \label{eq:ostokeslet}
  \GG = - \II \Delta \Gbh + \nabla \otimes \nabla \Gbh
\end{equation}
is then the analog of a Stokeslet
\cite{Pozrikidis1992} for \eqref{eq:ostokes}.

A related object is the stresslet, which is defined
in terms of the stress tensor of the velocity, pressure
pair induced by a Stokeslet. For these tensors, we find
that it is more convenient to express them in index notation
with the Einstein index summing convention.
Recall that the stress tensor $\bsigma$ is defined as 

\begin{equation}
  \sigma_{ij} = -p \delta_{ij} + \left ( \partial_{x_j}u_i
  +\partial_{x_i} u_j \right ) \; , \nonumber
\end{equation}
where $\delta_{ij}$ is the standard Kronecker delta notation.
The stresslet $\TT$ is defined to be

\begin{align}
  T_{ij\ell} &= - \partial_{x_j} \Glap \delta_{i\ell}
  + \partial_{x_\ell} \left ( -\Delta \Gbh \delta_{ij} +
  \partial_{x_i} \left(\partial_{x_j} \Gbh \right) \right)
  \nonumber \\
  & \qquad+ \partial_{x_i} \left ( -\Delta \Gbh \delta_{\ell j} +
  \partial_{x_\ell} \left(\partial_{x_j} \Gbh \right) \right)
  \; . \label{eq:ostress} 
\end{align}
Let $u_i = G_{ij} f_j$ and $p = \partial_{x_i} \Glap f_i$ be a
solution of the Stokes equations induced by a Stokeslet.
Then the corresponding stress tensor is given by
$\sigma_{i\ell} = T_{ij\ell} f_j$.

%For the layer potentials of the next section, the following
%formulas are useful. Let $\bnu$ be a given vector. When
%summing over the third index, we obtain
%
%\begin{equation}
%  \TT_{\cdot,\cdot,\ell} \nu_\ell = -\bnu \otimes \nabla \Glap
%  + \partial_\nu \left ( -\Delta \Gbh \II
%  + \nabla \otimes \nabla \Gbh \right)
%  + \nabla \otimes \left ( -\Delta \Gbh \bnu
%  + \partial_{\nu} \nabla \Gbh \right) \; . \nonumber
%\end{equation}
%Let $\btau = \bnu^\bot$. Then
%
%\begin{equation}
%  \TT_{\cdot,\cdot,k} \nu_k = -\bnu \otimes \nabla \Glap
%  - \nabla^\bot \otimes \nabla^\bot \partial_\nu \Gbh
%  +\nabla \otimes \nabla^\bot \partial_\tau \Gbh \; .
%  \nonumber
%\end{equation}

\subsubsection{Layer potentials}

We now use the Stokeslet and stresslet 
to define the single
and double layer potentials for the oscillatory Stokes problem.
For $\xx \in \R^2$, the single layer potential with density $\bmu$
is defined to be

\begin{equation} \label{eq:singlelayer}
  \bS [\bmu] (\xx) = \int_\Gamma \GG (\xx,\yy) \bmu(\yy)
  \, dS(\yy) \; .
\end{equation}
We use the notation $\bsigma_\bS[\bmu]$ to denote the
stress tensor of the single layer at any given point
$\xx \in \R^2 \setminus \Gamma$.

For $\xx \in \R^2 \setminus \Gamma$, the double layer
potential with density $\bmu$ is defined to be

\begin{equation} \label{eq:doublelayer}
  \bD [\bmu] (\xx) = \int_\Gamma \left ( \TT_{\cdot,\cdot,\ell}(\xx,\yy)
  \nu_\ell(\yy)\right )^\intercal \bmu(\yy) \, dS(\yy) \; ,
\end{equation}
where $\bnu$ denotes the outward unit normal to the boundary.
If we write $\bmu = \bnu \mu_\nu + \btau \mu_\tau$,
where $\btau = \bnu^\bot$ is the positively oriented unit
tangent to the curve, then we have

\begin{align} \label{eq:stokesdlkernel}
  \left ( \TT_{\cdot,\cdot,\ell}(\xx,\yy)\nu_\ell(\yy) \right )^\intercal
  \bmu(\yy) &= \left ( - \nabla \Glap(\xx,\yy) + 2 \nabla^\bot
  \partial_{\nu\tau} \Gbh(\xx,\yy) \right ) \mu_\nu(\yy) \nonumber \\
  & \qquad+
  \nabla^\bot \left (\partial_{\tau\tau}-\partial_{\nu\nu} \right )
  \Gbh(\xx,\yy) \mu_\tau(\yy) \; .
\end{align}

Let $\cS[\bmu]: C(\Gamma) \to C(\Gamma)$, and
$\cD[\bmu]: C(\Gamma) \to C(\Gamma)$ 
denote the restrictions of the layer potentials 
$\bS[\bmu]$ and $\bD[\bmu]$ on the boundary $\Gamma$, i.e.
for $\bx \in \Gamma$, 

\begin{equation}
  \cS [\bmu] (\xx) = \int_\Gamma \GG (\xx,\yy) \bmu(\yy)
  \, dS(\yy)
\end{equation}
and
\begin{equation}
\label{eq:dlformula}
  \cD [\bmu] (\xx) = \pv \int_\Gamma \left ( \TT_{\cdot,\cdot,\ell}(\bx,\by)
  \nu_\ell(\yy)
  \right )^\intercal \bmu(\yy) \, dS(\yy) \; ,
\end{equation}
where the \pv indicates that the integral is to be
evaluated in the principal value sense. 

For two vector valued
functions $\ff$ and $\bg$ defined on $\Gamma$, consider the bilinear
form
\begin{equation} \label{eq:bi_form}
  \langle \ff , \bg \rangle = \int_\Gamma \ff \cdot \bg dS \; .
\end{equation}
The definition of the adjoint used throughout the paper will be
the one induced by this form.

The adjoint of $\cD$ with respect to the above bilinear form
is of particular interest; and is given by
\begin{equation}
  \cD^{\intercal} [\bmu](\bx) = 
  \pv \int_\Gamma \left ( \TT_{\cdot,\cdot,\ell}(\bx,\by)\nu_\ell(\xx)
  \right ) \bmu(\yy) \, dS(\yy) \; .
\end{equation}

In the following lemma, we review the limiting values of
the layer potentials $\bS$ and $\bD$ on the boundary $\Gamma$.

\begin{lem}[Jump conditions] \label{lem:jump-conds}
  Suppose that $\Omega$ is a bounded region with a $C^{2}$ boundary
  $\Gamma$.
  Let $\bnu(\bx)$ denote the outward pointing normal at $\bx \in \Gamma$.
  Suppose that $\bmu \in C(\Gamma)$.
  Then $\bS[\bmu]$
  is continuous across $\Gamma$, and the exterior and interior
  limits of the surface traction of $\bD[\bmu]$ are equal.
  Furthermore, for $\bx_{0} \in \Gamma$,

  \begin{align}
    \lim_{h \downarrow 0^{+}} \bsigma_\bS[\bmu](\xx_0 \pm h\bnu(\xx_0)) \cdot \bnu(\xx_0)
    &= \mp \frac{1}{2} \bmu(\xx_0) + \cDt[\bmu](\xx_0) \\
    \lim_{h \downarrow 0^{+}} \bD[\bmu](\xx_0 \pm h\bnu(\xx_0)) 
    &= \pm \frac{1}{2} \bmu(\xx_0) + \cD[\bmu](\xx_0)    \; .
  \end{align}
\end{lem}

The above expressions are derived by noting that the
leading order singularity of these integral kernels
is the same as for the original Stokes case, so that
the standard jump conditions for Stokes
\cite{KimSangtae1991,Pozrikidis1992}
apply. 

\begin{lem} \label{lem:compact-sd}
  Suppose that $\Omega$ is a bounded region with a $C^{2}$ boundary
  $\Gamma$. Then the operators $\cS$ and $\cD$ defined
  above are compact operators on $C(\Gamma)\times C(\Gamma)$
  and $\mathbb{L}^2(\Gamma)\times \mathbb{L}^2(\Gamma)$.
\end{lem}

Compactness is proved by considering the
asymptotic expansion of each kernel about
$\xx=\yy$ and noting that each is at most
weakly singular.

\subsubsection{Representation Theorem}

In the following theorem, we sketch the proof of the equivalent of the
Green's identity for oscillatory Stokes setting, which is well-known
in the Stokes and modified Stokes settings
\cite{Pozrikidis1992,biros2002embedded,ladyzhenskaya1969mathematical}.

\begin{thrm} \label{thrm:rep-theorem}
  Let $\Omega$ be a bounded domain with $C^2$ boundary and let
  the pair $(\bu,p)$ satisfy the oscillatory Stokes equations
  \cref{eq:ostokes} in $\Omega$. Let $\bt$ denote the surface
  traction associated with $(\bu,p)$. Then

  \begin{equation} \label{eq:rep-theorem}
    \bS [\bt](\xx) - \bD[\bu](\xx) = \begin{cases} 
    \bu(\xx) &\quad \xx \in \Omega \,  \\
    0 &\quad \xx \in E 
    \end{cases} \; ,
  \end{equation}
  where $E=\R^2\setminus\bar\Omega$ is the
  exterior of the domain.
\end{thrm}

\begin{proof}
  Suppose that $\xx \in \Omega$.
  By the definitions of $\GG$ and $\Glap$, we have
  \begin{equation*}
    \bu(\xx) = \int_\Omega -(\Delta + k^2) \GG(\xx,\yy) \bu(\yy)
    + \nabla \otimes \nabla \Glap(\xx,\yy) \bu(\yy) \, dV(\yy) \; .
  \end{equation*}
  Applying Green's identity and the divergence theorem, we
  obtain
  \begin{equation}
    \bu(\xx) = \int_\Gamma \GG(\xx,\yy) \partial_\nu \bu
    - \partial_\nu \GG(\xx,\yy) \bu
    - \nabla \Glap(\xx,\yy) (\bnu \cdot \bu) \, dS 
    - \int_\Omega \GG(\xx,\yy) (\Delta + k^2) u \, dV \; .
    \nonumber
  \end{equation}
  Substituting the definition of the PDE and applying the divergence
  theorem again, we obtain
  \begin{equation}
    \bu(\xx) = \int_\Gamma \GG \partial_\nu \bu - p \GG \bnu - \partial_\nu \GG \bu
    - \nabla \Glap (\bnu \cdot \bu) \, dS  \; . \label{eq:rep_proof_1}
  \end{equation}
  From the divergence theorem and the divergence-free properties of
  $\uu$ and $\GG$, we then get

  \begin{equation}
    \int_\Gamma \GG \nabla (\bu \cdot \bnu)
    - (\nabla \GG \bnu)^\intercal \bu \, dS = 0 \; .  \label{eq:rep_proof_2}
  \end{equation}
  Adding \cref{eq:rep_proof_1,eq:rep_proof_2}, we get the desired
  result. The argument for the case $\xx \in E$ is similar.
\end{proof}

A consequence of \cref{thrm:rep-theorem} and the analyticity
of $\Gbh$ is 
\begin{cor}
\label{cor:analytic}  
  Let $(\uu,p) \in A(\Omega)$ be a solution of \cref{eq:ostokes}.
  Then each component of $\uu(\xx)$ is an analytic function
  of the coordinates $\xx$ in $\Omega$.
\end{cor}

The proof of \cref{cor:analytic} follows the same reasoning as
that for the Helmholtz case; see \cite[Theorem 3.5]{colton1983integral}.

\subsubsection{Null-space correction \label{subsubsec:nullspacecorr}}

Without modification, the standard layer potentials
can result in rank-deficient representations for
the boundary value problems. The nature of this deficiency
is treated in~\cref{sec:analysis} but for now we introduce a standard
operator used to correct this. For any integrable density
$\bmu$, let $\cW[\bmu]$ be defined by

\begin{equation} \label{eq:ones_operator}
  \cW[\bmu](\xx) = \frac{1}{|\Gamma|} \int_\Gamma \bnu(\xx)
  \left ( \bnu(\yy) \cdot \bmu(\yy) \right )
  \, dS(\yy) \; ,
\end{equation}
for any $\xx \in \Gamma$. We have

\begin{lem}
  \label{lem:propnullspacecorr}

  Let $\Omega$ be a domain with $C^2$ boundary and $\bmu$ be
  an integrable function defined on $\Gamma$. Then
  \begin{itemize}
  \item $\cW[\cW[\bmu]] = \cW[\bmu]$,
  \item $\cW^\intercal = \cW$,
  \item $\cW[\bmu - 2 \cD[\bmu]] = 0$,
  \item $\cW[\cS[\bmu]] = 0$,
  \end{itemize}
  where the transpose is induced by the bilinear
  form \cref{eq:bi_form}.
\end{lem}

\begin{proof}
  The first two results follow from the definitions of $\cW$ and
  the normal vector. The other two follow from the fact that
  $\bS[\bmu]$ and $\bD[\bmu]$ are divergence-free and
  an application of \cref{lem:jump-conds}.
\end{proof}

\begin{remark}
The operators $\bS,\bD,\cD$ and $\cD^{\intercal}$
depend on the Helmholtz parameter $k$ of the oscillatory Stokes equation.
In places where it is essential to highlight this dependence, in a slight
abuse of notation, we will use the symbols $\bS_{k}, \bD_{k},\cD_{k}$ and $\cDt_{k}$ to 
denote this dependence.
Similarly, we will use $A^{\Gamma}$ instead of the operator $A$ to highlight 
the dependence of the operator $A$ on the boundary of the region $\Gamma$.
\end{remark}

% flatex input end: [02prelim.tex]

%
% flatex input: [03analysis.tex]
\section{Fredholm analysis of the integral representations}
\label{sec:analysis}
In this section, we establish how layer potential
representations of oscillatory Stokes velocity fields
can be used to compute Stokes eigenvalues.
The main results show that for certain representations
the resulting integral equation is not invertible precisely
when $k^2$ is an eigenvalue.
For the interior Dirichlet eigenvalue problem, this is
proved separately for a double layer representation on
simply connected domains in \cref{thm:dlmain} and
for a combined-field representation on multiply connected
domains in \cref{thm:cfmain}.

Before proving the main theorems, we require a number
of uniqueness results for oscillatory Stokes boundary
value problems.
To prove the uniqueness results, we follow the 
structure presented in Colton and Kress~\cite[Ch. 3]{colton1983integral}
for the scalar Helmholtz equation.
While uniqueness results for interior Dirichlet,
Neumann, and impedance problems follow from energy
considerations and compactness arguments, the proofs
for the uniqueness of exterior problems are more
involved.
In particular, the exterior problems are only 
well-posed after imposing an appropriate radiation
condition.
We formulate well-posed boundary value problems 
for both interior and exterior domains with
Dirichlet, Neumann and impedance boundary
conditions and present uniqueness results for
each.

Along with the Fredholm alternative, these uniqueness
results are sufficient to prove
\cref{thm:dlmain,thm:cfmain}.
Once the main theorems are established,
the details of how to use the Fredholm determinant
as a numerical tool for computing the Stokes eigenvalues
follow in a straightforward manner from the results
in~\cite{zhao2015robust}.
We reproduce these results in the present
context for completeness.

\subsection{Boundary value problems --- interior}

Let $\Omega$ be a bounded domain with a $C^2$ boundary
denoted by $\Gamma$.
We summarize the interior Dirichlet, Neumann and impedance
boundary value problems in~\cref{def:int_dir,def:int_neu,def:int_imp} below.

\begin{definition}[Interior Dirichlet problem]
  \label{def:int_dir}
  Let $\ff \in C(\Omega)$ be given. Find $(\bu,p) \in A(\Omega)$
  such that
  \begin{equation}
  \begin{aligned} \label{eq:dir_interior}
    \Delta \bu + k^{2} \bu &= \nabla p \quad \bx \in \Omega \, ,\\
    \nabla \cdot \bu &= 0 \quad \bx \in \Omega \, ,  \\
    \bu &= \ff \quad \bx \in \Gamma \, .
  \end{aligned}
  \end{equation}
\end{definition}
Note that the divergence-free constraint for the oscillatory
Stokes equations implies a compatibility condition on the
Dirichlet data $\ff$, namely that
\begin{equation} \label{eq:dir_compat}
  \int_\Gamma \ff \cdot \bnu \, dS = 0 \; .
\end{equation}

\begin{definition}[Interior Neumann problem]
  \label{def:int_neu}
  Let $\bg \in C(\Omega)$ be given. Find $(\bu,p) \in A(\Omega)$
  such that
  \begin{equation}
  \begin{aligned} \label{eq:neu_interior}
    \Delta \bu + k^{2} \bu &= \nabla p \quad \bx \in \Omega \, ,\\
    \nabla \cdot \bu &= 0 \quad \bx \in \Omega \, ,  \\
    \bt &= \bg \quad \bx \in \Gamma \, .
  \end{aligned}
  \end{equation}
\end{definition}

\begin{definition}[Interior impedance problem]
\label{def:int_imp}
  Let $\bh \in C(\Omega)$ be given and suppose 
  $\eta \in \mathbb{C}$ with $\Re{(\eta)} >0$ and $\Im{(\eta)}\ge 0$. 
  Find $(\bu,p) \in A(\Omega)$  such that
  \begin{equation}
  \begin{aligned} \label{eq:imp_interior}
    \Delta \bu + k^{2} \bu &= \nabla p \quad \bx \in \Omega \, ,\\
    \nabla \cdot \bu &= 0 \quad \bx \in \Omega \, ,  \\
    \bt - i \eta \bu &= \bh \quad \bx \in \Gamma \, .
  \end{aligned}
  \end{equation}
\end{definition}

\subsection{A radiation condition for the oscillatory Stokes
  equation}

\begin{figure}
  \begin{center}
  \includegraphics[width=0.4\textwidth]{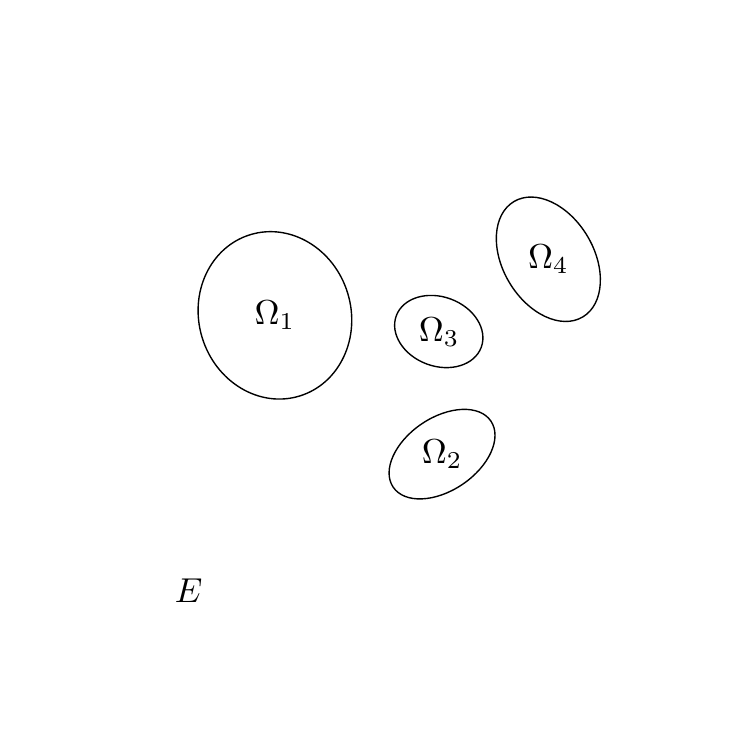}
  \caption{Example of an exterior domain with four obstacles.}
  \label{fig:ext_dom}
  \end{center}
\end{figure}

Let $\Omega$ be the union of a finite collection of
simply connected domains, i.e. $\Omega = \bigcup_{i=1}^m \Omega_i$
for some $m \in \N$,
and let $E = \R^{2} \setminus \bar{\Omega}$ denote its
exterior; see \cref{fig:ext_dom} for an example with $m=4$.
Let $\Gamma = \partial E$ denote the boundary of $E$ and
$\bnu(\yy)$ denote the exterior normal to the point $\yy$ on
$\Gamma$, i.e. the normal vector pointing out of $E$ into $\Omega$.
For a given function $\ff$ defined on $\Gamma$,
the exterior Dirichlet boundary value problem is to
find a pair $(\bu,p)$ which satisfies:
\begin{equation}
\begin{aligned}
\Delta \bu + k^{2} \bu &= \nabla p \quad \bx \in E \, ,\\
\nabla \cdot \bu &= 0 \quad \bx \in E \, ,  \\
\bu &= \ff \quad \bx \in \Gamma \, . \nonumber
\end{aligned}
\end{equation}
In addition to the boundary condition on
$\Gamma$, we must impose radiation conditions
at $\infty$, analogous to the Helmholtz equation.

Let $B_r(0)$ denote the disc of radius $r$ centered
at the origin and $\partial B_r(0)$ its boundary.
We propose the following radiation condition.

\begin{definition} \label{def:radcond}
Let $(\bu,p)$ satisfy the oscillatory Stokes equations in
the exterior of a bounded domain. We say that
the pair $(\bu,p)$ is {\em radiating} if
\begin{equation}
\lim_{r\to \infty} \sqrt{r} \left| \bt - i k \bu \right| \to 0 \, ,
\label{eq:radcond}
\end{equation}
uniformly in direction where $\bt = \bsigma \cdot \bnu$
with $\bnu = \xx/|\xx|$, i.e. $\bt$ is the surface
traction on $\partial B_r(0)$.     
\end{definition}

In the following lemma, we show that the oscillatory
Stokeslet satisfies the radiation condition. 
\begin{lem}
The oscillatory Stokeslet, as defined in \eqref{eq:ostokeslet}, 
satisfies the radiation condition in Definition \ref{def:radcond}.
\end{lem}
\begin{proof}
Consider the Stokeslet induced by an arbitrary charge
$k^2 \bpsi$ at the origin where $\bpsi \in \mathbb{C}^2$ 
is a constant. Let $r = |\xx|$,
$\bnu(\xx) = \xx/|\xx|$, and $\btau(\xx) = \bnu(\xx)^\perp$. We have

\begin{align*}
\bu (\xx) &= k^2 \GG(\xx,0) \bpsi \\
&= k^2 \left (-\II \Delta \Gbh(\xx,0)
+ \nabla \otimes \nabla \Gbh(\xx,0)\right ) \bpsi \\
&= -k^2 \left (\nabla^\perp \otimes \nabla^\perp \Gbh(\xx,0) \right) \bpsi \\
&= \left(\nabla^\perp \otimes \nabla^\perp \left ( \frac{1}{2\pi}
\log r + \frac{i}{4} H_0^{(1)} (kr) \right ) \right)\bpsi \; .
\end{align*}
Note that derivatives of $\log r$ are $\littleo (1/\sqrt{r})$
and that the pressure associated with the Stokeslet is
$p = \nabla \Glap(\xx) \cdot \bpsi$. We then have

\begin{align*}
\left | \sigma \cdot \bnu(\xx) - i k \bu \right | &=
\left | p \bnu(\xx) + \partial_{{\nu_x}} \bu + \nabla (\bu \cdot \bnu(\xx))
- ik \bu \right | \\
&\leq \left | \partial_{{\nu_x}} \bu - ik\bu \right | + \left | \nabla(\bu \cdot \bnu(\xx)) \right |
+ \littleo (1/\sqrt{r}) \\
&\leq \frac{1}{4} \left | \partial_{{\nu_x}} \left(\nabla^\perp \otimes
\nabla^\perp \left (H_0^{(1)} (kr) \right ) \right)\bpsi
- i k \left(\nabla^\perp \otimes \nabla^\perp
\left (H_0^{(1)} (kr) \right ) \right)\bpsi \right | \\
& \quad + \left | \nabla \left( \partial_{\tau_x} \left ( 
\nabla^\perp \left (H_0^{(1)} (kr) \right ) \cdot \bpsi  \right )
\right) \right | + \littleo (1/\sqrt{r}) \; .
\end{align*}
Because $H_0^{(1)}(kr)$ has the asymptotic expansion 

\begin{equation}
H_0^{(1)}(kr) = \sqrt{\frac{2}{\pi k r}} e^{i(rk-\pi/4)} \left ( 1 + O\left (
\frac{1}{r} \right ) \right ) \; \nonumber
\end{equation}
as $r\to \infty$, we have

\begin{equation}
\left | \partial_{{\nu_x}} \left(\nabla^\perp \otimes
\nabla^\perp \left (H_0^{(1)} (kr) \right ) \right)\bpsi
- ik \left(\nabla^\perp \otimes \nabla^\perp
\left (H_0^{(1)} (kr) \right ) \right)\bpsi \right | =
o ( 1/\sqrt{r} ) \; . \nonumber
\end{equation}
Finally, since $H_0^{(1)}(kr)$ is radially symmetric,
we have

\begin{equation}
\left | \nabla \left( \partial_{\tau_x} \left ( 
\nabla^\perp \left (H_0^{(1)} (kr) \right ) \cdot \bpsi  \right )
\right) \right | = 0 \; , \nonumber
\end{equation}
so that the Stokeslet satisfies the radiation condition.
\end{proof}
A consequence of the above lemma is that the oscillatory Stokes
single layer potential satisfies the radiation condition.
\begin{cor}
Suppose that $\Gamma$ is the boundary of a region $\Omega$
and is $C^{2}$. 
Suppose that $\bmu \in C(\Gamma)\times C(\Gamma)$, then
the oscillatory Stokes single layer potential $\bS[\bmu]$,
as defined in~\cref{eq:singlelayer}, satisfies the radiation condition.
\end{cor}
Unfortunately, the stresslet, as defined in~\cref{eq:ostress}, 
does not necessarily satisfy the radiation condition.
The reason for failure is the logarithmic growth of 
the pressure at $\infty$.
However, this turns out to be a rank-one issue and 
the oscillatory Stokes double layer potential
does satisfy the radiation condition 
if the density satisfies an integral constraint.
The following lemma proves this result.

\begin{lem}
Suppose that $\Gamma$ is the boundary of a region $\Omega$
and is $C^{2}$. 
Suppose that $\bmu \in C(\Gamma) \times C(\Gamma)$ and satisfies
$\int_{\Gamma} \bmu \cdot \bnu dS = 0$, where
$\bnu$ denotes the outward normal to the curve $\Gamma$.
Then, the oscillatory Stokes 
double layer potential $\bD[\bmu]$, as defined in \eqref{eq:doublelayer},
also satisfies the radiation condition.
\end{lem}

\begin{proof}
We only establish the
decay of the pressure, which we will denote by $p^\bD$;
the rest of the terms in \eqref{eq:radcond} can be bounded
using an argument like that for the Stokeslet above.
Because $p^\bD$ is harmonic in the exterior of any disc
containing $\Gamma$, it is sufficient to show that
$|\nabla p^\bD| = \bigo (1/r^2)$. Let $\btau$ denote the
positively oriented tangent to the curve $\Gamma$ and
$\mu_\nu(\yy)$ and $\mu_\tau(\yy)$ denote $\bmu(\yy) \cdot \bnu(\yy)$ and
$\bmu(\yy) \cdot \btau(\yy)$, respectively. Substituting
$\bD \bmu$ into \eqref{eq:ostokes}, we obtain

\begin{align*}
\nabla p^\bD(\xx) &= (\Delta + k^2) \bD \bmu(\xx) \\
&= \int_\Gamma \left (-k^2 \nabla \Glap (\xx,\yy) +
2 \nabla^\perp \partial_{\nu\tau} \Glap (\xx,\yy) \right ) \mu_\nu(\yy)
\, dS(\yy) \\
& \qquad + \int_\Gamma \nabla^\perp (\partial_{\tau\tau}-\partial_{\nu\nu})\Glap \mu_\tau(\yy)
\, dS(\yy) \; .
\end{align*}
The other terms are higher-order derivatives
of $\Glap$, so it is sufficient to show that the term

\begin{align*}
|\nabla p_1(\xx)| &:= \left |-k^2 \nabla \int_\Gamma \Glap(\xx,\yy)
\mu_\nu(\yy) \, dS(\yy) \right |
\end{align*}
is $\bigo (1/r^2)$. In the following, let $z = x_1 + i x_2$ be the
point corresponding to $\xx$ in the complex plane and let $R$
be the radius of some disc containing $\Gamma$. If $|\xx| > 2R$,
we can use the standard multipole expansion of $\log(z-(y_1+iy_2))$
and the assumption that $\int_\Gamma \bmu \cdot \bnu = 0$
to obtain
\begin{align*}
|\nabla p_1(\xx)| &= \frac{k^2}{2\pi} \left |  \partial_z \int_\Gamma \log(z-(y_1+iy_2))
\mu_\nu(\yy) dS(\yy) \right | \\
&= \frac{k^2}{2\pi} \left |  \partial_z  \left ( \log(z) \int_\Gamma \mu_\nu(\yy) \, dS(\yy)
+ \sum_{l=1}^\infty \frac{1}{z^l} \int_\Gamma \left( y_1+iy_2 \right)^l \mu_\nu(\yy) \, dS(\yy)
\right ) \right | \\
&= \frac{k^2}{2\pi} \left | \sum_{l=1}^\infty \frac{-l}{z^{l+1}}
\int_\Gamma \left( y_1+iy_2 \right)^l \mu_\nu(\yy) \, dS(\yy) \right | \\
&= \bigo (1/r^2) \; .
\end{align*}
\end{proof}

\subsection{Boundary value problems --- exterior}

Let $E$ and $\Gamma$ be as in the previous subsection.
The radiation condition allows for a well-posed formulation
of the exterior boundary value problems, which we
summarize in
\cref{def:dir_exterior,def:neu_exterior,def:imp_exterior} below.

\begin{definition}[Exterior Dirichlet problem]
  \label{def:dir_exterior}
  Let $\ff \in C(\Gamma)$ be given. Find $(\bu,p) \in A(E)$
  such that
  \begin{equation}
  \begin{aligned} \label{eq:dir_exterior}
    \Delta \bu + k^{2} \bu &= \nabla p \quad \bx \in E \, ,\\
    \nabla \cdot \bu &= 0 \quad \bx \in E \, ,  \\
    \bu &= \ff \quad \bx \in \Gamma \, , 
  \end{aligned}
  \end{equation}
  and $(\bu,p)$ satisfies the radiation condition in
  \cref{def:radcond}.
\end{definition}
\begin{definition}[Exterior Neumann problem]
  \label{def:neu_exterior}  
  Let $\bg \in C(\Gamma)$ be given. Find $(\bu,p) \in A(E)$
  such that
  \begin{equation}
  \begin{aligned} \label{eq:neu_exterior}
    \Delta \bu + k^{2} \bu &= \nabla p \quad \bx \in E \, ,\\
    \nabla \cdot \bu &= 0 \quad \bx \in E \, ,  \\
    \bt &= \bg \quad \bx \in \Gamma \, ,
  \end{aligned}
  \end{equation}
  and $(\bu,p)$ satisfies the radiation condition in
  \cref{def:radcond}.
\end{definition}

\begin{definition}[Exterior impedance problem]
  \label{def:imp_exterior}  
  Let $\bh \in C(\Gamma)$ be given and suppose that
  $\eta \in \mathbb{C}$ with $\Re{(\eta)} > 0$ and $\Im{(\eta)} \ge 0$. 
  Find $(\bu,p) \in A(E)$
  such that
  \begin{equation}
  \begin{aligned} \label{eq:imp_exterior}
    \Delta \bu + k^{2} \bu &= \nabla p \quad \bx \in E \, ,\\
    \nabla \cdot \bu &= 0 \quad \bx \in E \, ,  \\
    \bt + i\eta \bu &= \bh \quad \bx \in \Gamma \, ,
  \end{aligned}
  \end{equation}
  and $(\bu,p)$ satisfies the radiation condition in
  \cref{def:radcond}.
\end{definition}

\subsection{Uniqueness results}

Before moving on to the exterior uniqueness theorems,
we establish the well-known result that
the $k$ corresponding to interior eigenvalues, $k^2$,
are real-valued.

\begin{thrm}
  Let $\Omega$ be a bounded domain and suppose that
  $\Im (k) \neq 0$. Then both the interior
  Dirichlet and Neumann boundary value problems have
  unique solutions.
\end{thrm}
\begin{proof}
  A couple applications of the divergence theorem establish
  that
  \begin{equation} \label{eq:greenlike}
    \int_\Omega |2\be(\bu)|^2 - \overline{k}^2 |\bu|^2 \, dV
    = \int_\Gamma \bu \cdot \overline{\bt} \, dS \; . 
  \end{equation}
  Suppose that either $\bu = 0$ or $\bt = 0$ on $\Gamma$.
  Then, the right hand side of \cref{eq:greenlike} is
  zero. 
  Taking the real and imaginary parts of \cref{eq:greenlike},
  it is clear that $\bu \equiv 0$, if $\text{Im}(k) \neq 0$.
\end{proof}

In the following lemma, we prove uniqueness for the
interior impedance problem.

\begin{thrm}
  Let $\Omega$ be a bounded domain and suppose that
  $\Re(k),\Im (k) > 0$. Then the interior
  impedance problem has a unique solution.
\end{thrm}
\begin{proof}
  Plugging $\bt = i\eta \bu$ in~\cref{eq:greenlike}
  and taking the imaginary part, we get
  \begin{equation}
   2 \Re{(k)} \Im{(k)} \int_{\Omega} |\bu|^2\, dV + 2\Re{(\eta)} \int_{\Gamma}
   |\bu|^2 \, dS = 0 \, ,
  \end{equation}
  from which it is is clear that $\bu \equiv 0$, since
  $\Re{(\eta)}, \Re{(k)}, \Im{(k)} > 0$.
\end{proof}

We now turn our attention to the proofs of the exterior boundary value problems
for the oscillatory Stokes equation. The following lemmas are useful for
proving these results.

\begin{lem}
  \label{lem:rep}
  Let the unbounded region $E$ be given as the exterior
  of a finite collection of bounded domains.
  Suppose that $(\bu,p)$ satisfies the oscillatory Stokes equation in 
  $E$ as well as the radiation condition~\cref{eq:radcond}. 
  Then 
  \begin{multline}
    \label{eq:repinfest}  
    \lim_{r\to\infty}
    \int_{|\by|=r} \left( |\bt|^2 + |k|^2 |\bu|^2 \right) dS +
    2 \Im(k) \int_{E \cap B_{r}(0)} \left(|k|^2 |\bu|^2 + |2\be(\bu)|^2 \right)
    dV \\
    + 2 \Im \left( k \int_{\Gamma} \bu \cdot 
\overline{\bt} dS  \right) = 0
  \end{multline}
  
\end{lem}

\begin{proof}
Since $(\bu,p)$ satisfies the radiation condition, we have that
\begin{equation}
\lim_{r\to\infty} \int_{|\by|=r} | \bt - i k \bu|^2 dS = 
\lim_{r\to\infty} \int_{|\by| =r} \left( |\bt|^2 + |k|^2|\bu|^2 + 2 \Im 
\left( k \bu\cdot \overline{\bt} \right) dS \label{eq:raddecayproof1}
\right) = 0 \, . 
\end{equation}
Since $\bu$ satisfies the oscillatory Stokes equation $E \cap B_{r}(0)$,
using a couple of applications of the divergence theorem, we have that
\begin{equation}
\int_{E\cap B_{r}(0)} |2 \be(\bu)|^2 dV =
-\int_{\Gamma} \bu \cdot \overline{\bt} dS
+ \int_{|\by|=r} \bu \cdot \overline{\bt} dS + \overline{k}^2 
\int_{E \cap B_{r}(0)} |\bu|^2 dV \,. \label{eq:raddecayproof2}
\end{equation}
Combining~\cref{eq:raddecayproof1,eq:raddecayproof2}, we get
\begin{multline*}
\lim_{r\to\infty} \int_{|\by|=r}\left(|\bt|^2 + |k|^2 |\bu|^2 \right) dS 
+ 2 \Im(k)\int_{E \cap B_{r}(0)} \left(|2\be(\bu)|^2 + |k|^2 |\bu|^2 
\right) dV \\
+ 2\Im \left ( k \int_{\Gamma} \bu \cdot \overline{\bt} dS \right) = 0 \, .
\end{multline*}
\end{proof}

In the next lemma, we prove the analogue of Rellich's lemma for the
oscillatory Stokes equation. 
\begin{lem}
  \label{lem:rellich}
  Let the unbounded region $E$ be given as the exterior
  of a finite collection of bounded domains.
  Suppose $k$ is real, $\bu$ satisfies the oscillatory
  Stokes equation in $E$, and that 
\begin{equation}
\lim_{r \to \infty} \int_{|\by|=r} |\bu|^2 dS = 0 
\, . \label{eq:decayatinf}
\end{equation}
Then each component of $\bu$ is harmonic in $E$.
\end{lem}
\begin{proof}
We first note that each component of $\bu= (u_{1},u_{2})$ satisfies the 
oscillatory biharmonic equation in $E$, i.e.
\begin{equation}
\Delta (\Delta + k^2) u_{j} = 0 \quad j=1,2 \,. \nonumber
\end{equation}
For $r$ sufficiently large, we can express $u_{j}$ in the Fourier basis as
\begin{equation}
u_{j}(r,\theta) = \sum_{n=-\infty}^{\infty} a_{j,n}(r) e^{i n \theta}  \quad 
j=1,2 \, . \nonumber
\end{equation}
Using Parseval's identity then
\begin{equation}
\int_{|\by|=r} |\bu|^2 dS = r\sum_{n=-\infty}^{\infty} |a_{1,n}(r)|^2  +
|a_{2,n}(r)|^2 \, . \nonumber
\end{equation}
Since $\bu$ satisfies~\cref{eq:decayatinf}, we conclude that
\begin{equation}
\lim_{r\to\infty} r|a_{j,n}(r)|^2 = 0 \quad j=1,2 \, , \label{eq:adecay}
\end{equation}
Since $u_{j}$, $j=1,2$ satisfies the oscillatory biharmonic equation,
the functions $a_{j,n}$ are linear combinations of 
\begin{equation}
r^{|n|}, r^{-|n|}, H^{(1)}_{n}(k r), H^{(2)}_{n}(k r) \, , \quad
n\neq 0 \, , \nonumber
\end{equation}
and
\begin{equation}
1, \log{(r)}, H^{(1)}_{0}(k r), H^{(2)}_{0}(k r) \quad n=0 \, ,  \nonumber
\end{equation} 
where $H_{n}^{(1),(2)}(\cdot)$ are the Hankel functions of the first and
second kind of order $n$.
Since $a_{j,n}(r)$ satisfy~\cref{eq:adecay}, and using the asymptotic 
expansion of $H_{n}^{(1),(2)}(kr)$ when $k$ and $r$ are real-valued, we note
that the projection of $a_{j,n}$ on $r^{|n|}$, and $H_{n}^{1,2}(k r)$
must be zero. Thus, for sufficiently large $r$,
\begin{equation}
u_{j}(r,\theta) = \sum_{n=-\infty}^{\infty} \frac{a_{j,n} e^{i n \theta}}{r^{|n|}} 
\, , \nonumber
\end{equation}
i.e. $u_{j}$ is harmonic 
in $B_{r}(0)^{c}$.
Finally, by \cref{cor:analytic}, $\bu$ is
analytic in $E$. Therefore, each $u_j$ is harmonic
throughout $E$.
\end{proof}

\begin{remark} \label{rmk:harmu}
  Note that if $\bu$ satisfies the assumptions
  of~\cref{lem:rellich}, then each component is harmonic
  and thus $\bu$ satisfies
\begin{align}
k^2 \bu &= \nabla p \label{eq:massconsred} \; .
\end{align}
\end{remark}

\begin{remark}
  It should be noted that in~\cref{lem:rellich}, $\bu$ need
  not be a radiating solution. All that is assumed of $\bu$
  is that it satisfies the oscillatory Stokes equations in $E$.
\end{remark}

We now have the results needed to establish the
uniqueness of exterior boundary value problems.

\begin{thrm}[Uniqueness of the Exterior Dirichlet Problem]
  \label{thrm:unique_dir_ext}
  Let the unbounded region $E$ be given as the exterior
  of a finite collection of bounded domains.
  Suppose that $\Im(k)\geq 0$ and 
  that $(\bu,p)$ is a radiating solution to the oscillatory Stokes
  equation in $E$ with $\bu =0$ on the boundary $\Gamma$, then
  $\bu \equiv 0$ in $E$.
\end{thrm}

\begin{proof}
Since $\bu = 0$ on $\Gamma$, it follows from~\cref{lem:rep} that
\begin{equation}
\lim_{r\to\infty}
\int_{|\by|=r} \left( |\bt|^2 + |k|^2 |\bu|^2 \right) dS +
2 \Im(k) \int_{E \cap B_{r}(0)} \left(|k|^2 |\bu|^2 + |\be(\bu)|^2 \right)
dV = 0 \nonumber
\end{equation}

Suppose that $\Im(k) > 0$. Then, it is immediate that
$\bu \equiv 0$ in $E$.

Suppose that $k$ is real valued. It is clear that
\begin{equation}
\lim_{r\to\infty} \int_{|\by|=r} |\bu|^2 dS = 0 \, . \nonumber
\end{equation}
Thus, the conditions on $\bu$ and $k$ in~\cref{lem:rellich}
are satisfied, and each component of $\bu$ is a harmonic function
with $\bu \to 0$ as $r \to \infty$. Furthermore, since $\bu=0$ on
$\Gamma$, by the uniqueness of solutions to the
Dirichlet problem for Laplace's equation
on exterior domains, we conclude that $\bu \equiv 0$
in $E$.
\end{proof}

\begin{thrm}[Uniqueness of the Exterior Neumann Problem]
  Suppose that $\Omega$ is the union of a finite collection 
  of simply connected domains, i.e. $\Omega = \bigcup_{i=1}^m \Omega_{i}$
  for some $m \in \N$, with $C^{2}$ boundaries, 
  and let $E = \R^{2} \setminus \bar{\Omega}$ denote its exterior;
  see~\cref{fig:ext_dom} for an example with $m=4$.
  Let $\Gamma_{i}$ denote the boundary of $\Omega_{i}$, 
  and $\Gamma = \bigcup_{i=1}^{m} \Gamma_{i}$ denote the boundary
  of $\Omega$.
  Suppose that $\Im(k)\geq 0$ and 
  that $(\bu,p)$ is a radiating solution to the oscillatory Stokes
  equation in $E$ with $\bt = 0$ on the boundary $\Gamma$, then
  $\bu \equiv 0$ in $E$.
\end{thrm}

\begin{proof}
Since $\bt = 0$ on $\Gamma$, it follows
from~\cref{eq:repinfest} that
\begin{equation}
\lim_{r\to\infty}
\int_{|\by|=r} \left( |\bt|^2 + |k|^2 |\bu|^2 \right) dS +
2 \Im(k) \int_{E \cap B_{r}(0)} \left(|k|^2 |\bu|^2 + |\be(\bu)|^2 \right)
dV = 0 \; . \nonumber
\end{equation}

Suppose that $\Im(k) > 0$. It is then immediate
that $\bu \equiv 0$ in $E$.

Suppose that $k$ is real. It is clear that
\begin{equation}
\lim_{r\to\infty} \int_{|\by|=r} |\bu|^2 dS = 0 \, . \nonumber
\end{equation}
Thus, the conditions on $\bu$ and $k$ in \cref{lem:rellich}
are satisfied, and each component of $\bu$ is a harmonic function
with $\bu \to 0$ as $r \to \infty$. Furthermore, as observed
in \cref{rmk:harmu}, $k^2 \bu = \nabla p$. Then, the boundary
condition becomes $0 = \bt = -p \bnu + 2 \nabla \partial_\nu p/k^2$.
Because $0 = \btau \cdot \bt = 2\partial_{\tau\nu} p/k^2$,
$\partial_\nu p  = c_{i}$ on $\Gamma_{i}$ for each $\Gamma_{i}$,
where $c_{i}$ is a constant. 
Observe that $|\bu|$ and $|\be(\bu)|$ must be $O(1/r)$ as
$r\to\infty$. Thus,
for a radiating pair $(\bu,p)$ with $p$ harmonic,
we have that $|p| = O(1/r)$ and $|\nabla p| = O(1/r^2)$
as $r\to\infty$.
Since the boundary is $C^{2}$ and the boundary data for $p$
is analytic, we conclude that $p$ is $C^{2}$ in
$\overline{E}$.
Furthermore, $\bt=0$ implies $p \bnu = 2\nabla \partial_{\nu} p/k^2$, 
and taking the dot
product with $\bnu$, we get
$p = 2\partial_{\nu \nu}p/k^2$. 
It then follows that
$p = -2\partial_{\tau \tau}p/k^2$ on $\Gamma$ since 
$p$ is harmonic in $E$ and $C^{2}$ in $\overline{E}$. 
Since $p$ satisfies the radiation condition at $\infty$,
we have
\begin{equation}
\begin{aligned}
\int_{E} |\nabla p|^2 dV &= 
\sum_{i=1}^{m} \int_{\Gamma_{i}} p \partial_{\nu}p \,dS \\
&= \sum_{i=1}^{m} c_{i} \int_{\Gamma_{i}} p \,dS \quad \text{(Since $\partial_{\nu} p = c_{i}$ on
$\Gamma_{i}$)} \\ 
&= -\sum_{i=1}^{m} \frac{2c_{i}}{k^2} \int_{\Gamma_{i}} \partial_{\tau \tau} p \,dS 
\quad \text{(Since $p = -\partial_{\tau \tau}p/k^2$ on $\Gamma$)} \\
&= 0
\end{aligned}
\end{equation}
Thus, $p$ is a constant in $E$. Furthermore, since $p\to 0$ at $\infty$, 
we conclude that $p\equiv 0$ in $E$. 
Finally, since $k^2 \bu = \nabla p$, we conclude that $\bu\equiv 0$ in 
$E$.
\end{proof}

\begin{thrm}[Uniqueness of the Exterior Impedance Problem]
  Let the unbounded region $E$ be given as the exterior
  of a finite collection of bounded domains.
  Suppose that the complex numbers $\eta$ and $k$ satisfy that
  $\Re(\eta), \Re(k) > 0$ and $\Im(\eta),\Im(k) \geq 0$.
  Suppose further that
  $(\bu,p)$ is a radiating solution of the
  oscillatory Stokes equation in $E$ which satisfies
  the homogeneous impedance boundary condition
  \begin{equation}
\bt + i \eta \bu = 0 \quad \xx \in \Gamma \, . \nonumber
\end{equation}
Then $\bu \equiv 0$ for $\xx \in E$.
\end{thrm}

\begin{proof}
Since $\bu$ satisfies the radiation condition at $\infty$ and $\bt = -i\eta \bu$
on $\Gamma$, it follows from~\cref{eq:repinfest} that
\begin{align*}
0 &=
\int_{|\by|=r} \left( |\bt|^2 + |k|^2 |\bu|^2 \right) dS +
2 \text{Im}(k) \int_{E \cap B_{r}(0)} \left(|k|^2 |\bu|^2 + |2\be(\bu)|^2 \right)
dV \nonumber \\
& \qquad + 2 \text{Im} \left( k \int_{\Gamma} \bu \cdot \overline{\bt} dS  \right) \\
&= 
\int_{|\by|=r} \left( |\bt|^2 + |k|^2 |\bu|^2 \right) dS +
2 \text{Im}(k) \int_{E \cap B_{r}(0)} \left(|k|^2 |\bu|^2 + |2\be(\bu)|^2 \right)
dV \nonumber \\
& \qquad + 2 \left( \left (\Re(k)\Re(\eta) + \Im(k)\Im(\eta)
\right ) \int_{\Gamma} |\bu|^{2} dS  \right)
\, .
\end{align*}
Because all of the quantities in the last expression above are
nonnegative, we have that
\begin{equation}
  \int_{\Gamma} |\bu|^{2} = 0 \implies \bu = 0  \quad \xx \in \Gamma \, .
  \nonumber
\end{equation}
The result then follows from the uniqueness of solutions to the exterior
Dirichlet problem.
\end{proof}

\subsection{The integral equations and their null-spaces}

\begin{figure}
\begin{center}
\includegraphics[width=0.4\textwidth]{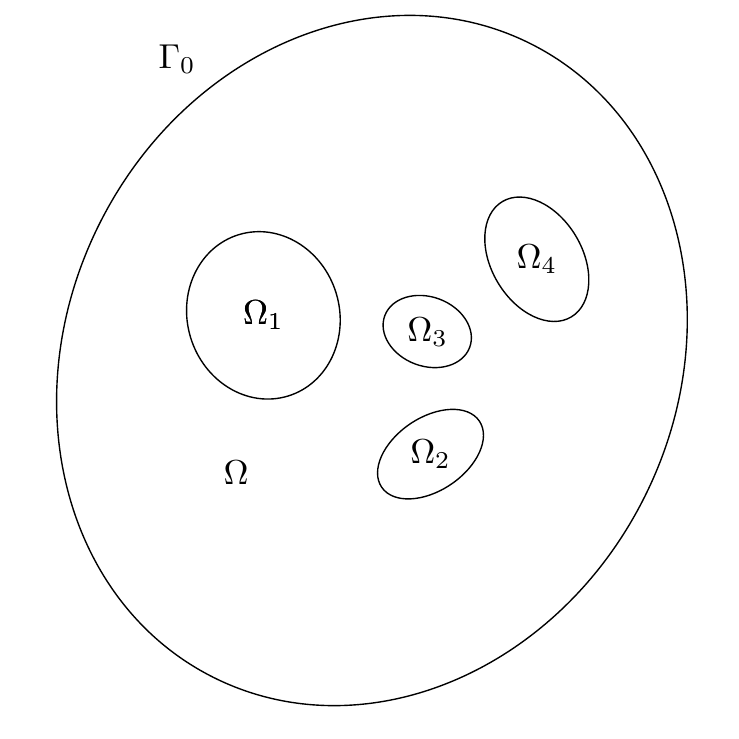}
\end{center}
\caption{Example of a multiply connected domain with four obstacles.}
\label{fig:mc_dom}
\end{figure}

In this section, we establish the correspondence between
Stokes eigenvalues and the invertibility of certain integral
equations arising from layer potential representations
of solutions to the oscillatory Stokes equation.

Let $\Omega$ be a bounded domain given as
the intersection of a simply connected domain $\Omega_0$ and
the exteriors of a finite collection of bounded,
simply connected domains $\{ \Omega_i \}_{i=1}^m$
whose closures are contained in $\Omega_0$;
see \cref{fig:mc_dom} for an example with four inclusions.
Note that the
exterior of $\Omega$, which we denote by $E$, is the
disjoint union of the exterior of $\Omega_0$, which we
denote $E_0$, with the sets $\{ \Omega_i \}_{i=1}^m$.
Let $\Gamma$ denote the boundary of $\Omega$ with the normal
$\bnu$ pointing out of $\Omega$.
We will use superscript $+$ and $-$ signs
to indicate the limit values of a function on $\Gamma$
as approached from the exterior and interior, respectively.

For the sake of brevity, we consider only the
Dirichlet eigenvalue problem but a similar analysis
could be applied to the Neumann eigenvalue problem.
We analyze two different representations for the
Dirichlet problem: a double layer potential,
i.e. setting $\bu=\bD_{k} \bmu$, and a combined-field
layer potential, i.e. setting
$\bu=(i\eta \bS_{k} + \bD_{k})\bmu$. 
While both representations result in a second kind
integral equation for the oscillatory Stokes Dirichlet
problem, the double layer potential has spurious non-trivial
nullspaces on domains with positive genus, as
explained below.

\begin{remark}
  The application of the Fredholm alternative here
  again follows the structure used for the Laplace
  eigenvalue problem in~\cite[Ch. 3]{colton1983integral}.
\end{remark}

\subsubsection{Dirichlet eigenvalues --- double layer
  representation}
\label{subsec:dlanalysis}
Suppose that the solution to the oscillatory
Stokes Dirichlet problem, \cref{eq:dir_interior},
is represented using a double layer potential defined
on $\Gamma$, i.e. setting $\bu = \bD_{k} \bmu$ where
$\bmu$ is an unknown density. 
Substituting this expression
into the boundary condition and applying
\cref{lem:jump-conds}, we obtain

\begin{equation}
  (\cI - 2\cD_{k}) \bmu = -2\ff \; . \label{eq:inteq_dir_int}
\end{equation}

The rank deficiency of $\cI-2\cDk$ is
well-known~\cite{biros2002embedded} and we summarize it
in the lemma below.

\begin{lem}
  \label{lem:nunull} In the notation above,
  $\bnu \in \cN(\cI - 2\cDkt)$.
\end{lem}
\begin{proof}
From~\cref{lem:propnullspacecorr}, we note that
$W[(\cI - 2\cD_{k})\bmu] =0$ implies that
$\left<(\cI -2\cD_{k})\bmu ,\bnu \right> = 0$ for
all $\bmu$, i.e. $\bnu \in R(\cI - 2\cD_{k})^{\perp}$, where
$R(A)$ denotes the range of the operator $A$. 
By the Fredholm alternative, the result then follows.
\end{proof}
Thus, we instead analyze the equation
\begin{equation}
(\cI - 2\cD_{k}  -2\cW) \bmu = -2\ff \; \, . \label{eq:inteq_dir_int_mod}
\end{equation}
Note that if $\ff$ satisfies the compatibility
condition $\int_\Gamma \ff \cdot \bnu \, dS =0$, then
\cref{eq:inteq_dir_int_mod} implies \cref{eq:inteq_dir_int}.

On simply connected domains, there is a one-to-one correspondence between
the eigenvalues of the Dirichlet problem for
the Stokes equation
and the values of $k$ for which the operator $(\cI - 2\cD_{k} - 2\cW)$
is not invertible. To prove this, we also need the 
following lemma:
\begin{lem}
  \label{lem:nutracli}
  If $\bt^{-}$ is the surface traction associated
  with an interior Dirichlet Stokes eigenfunction $\bu$,
  then $\bt^{-}$ and $\bnu$ are linearly independent.
\end{lem}
\begin{proof}
We first note that $\bS_{k}[\bnu](\bx) = 0$ for all $\bx \in \Omega$.
This follows from an application of the divergence
theorem and the fact that
oscillatory Stokeslet is divergence free in $\Omega$. 
If $\bt^{-}$ is a surface traction associated with
a Stokes eigenvalue, then
using Green's theorem, it follows that
$\bS_{k}[\bt^{-}](\bx) = \bu(\bx) \neq 0$ for $\bx \in \Omega$ and
thus $\bt^{-}$ and $\bnu$ are linearly independent.
\end{proof}

\begin{thrm}
\label{thm:dlmain}
Suppose that $\Omega$ is a bounded, simply connected domain. Then, the operator
$(\cI - 2\cD_{k} - 2 \cW)$ is not invertible if and only if $k^2$ is a
Dirichlet eigenvalue for the Stokes equation on $\Omega$.
\end{thrm}
\begin{proof}
  Suppose that $k^2$ is not a Dirichlet eigenvalue for
  the Stokes equation on
$\Omega$. 
Suppose further that $\bmu$ satisfies
\begin{equation}
(\cI - 2\cD_{k} - 2\cW) \bmu = 0 \, , \label{eq:dlproofrep}
\end{equation}
i.e. $\bmu$ is in the null-space
of $(\cI - 2 \cD_{k} - 2 \cW)$. 
Applying the operator $\cW$ to~\cref{eq:dlproofrep} and 
using~\cref{lem:propnullspacecorr}, we get
\begin{equation}
0 = \cW [(\cI - 2\cD_{k} - 2\cW)\bmu] = -2\cW [\bmu] \, .
\end{equation}
Thus~\cref{eq:dlproofrep} reduces to
\begin{equation}
(\cI - 2 \cD_{k})\bmu = 0\, .
\end{equation}
Suppose now $\bu = -2\bD_{k}[\bmu]$ in $\Omega$.
Then $\bu$ is a solution to the oscillatory Stokes equation in $\Omega$,
and applying~\cref{lem:jump-conds}, we get that the interior
limit of the velocity $\bu^{-}=(\cI - 2\cD_{k}) \bmu = 0$
on $\Gamma$. Since $k^2$ is not a Dirichlet eigenvalue for the Stokes
equation on $\Omega$, we conclude that $\bu \equiv 0$ in $\Omega$. 
In particular, this implies that the interior limit of the surface traction
denoted by $\bt^{-} = 0$ on $\Gamma$. 
Using~\cref{lem:jump-conds} again, we conclude that the exterior limit
of the surface traction, $\bt^{+}$, is $0$ on $\Gamma$. 
Note that $\bu$ is a radiating solution of the oscillatory
Stokes equation in the exterior $E$, as $\cW[\bmu] = 0$ implies
$\int_{\Gamma} \bmu \cdot \bnu = 0$.
From the uniqueness of solutions to the exterior Neumann problem, 
we conclude that $\bu \equiv 0$ in $E$ as well,
which in particular implies that
the exterior limit of the velocity, $\bu^{+}$, is $0$ on $\Gamma$.   
Using the jump conditions in~\cref{lem:jump-conds} again, we
get that $2\bmu = \bu^{-} - \bu^{+} = 0$.
Thus $(\cI - 2\cD_{k} - 2\cW)$ is invertible if $k^2$
is not a Dirichlet eigenvalue for the Stokes equation on $\Omega$.

To prove the converse, note that from \cref{thrm:rep-theorem} we have 
  \begin{equation} 
    \bS_{k} [\bt](\xx) - \bD_{k}[\bu](\xx) = \begin{cases} 
    \bu(\xx) &\quad \xx \in \Omega \, , \\
    0 &\quad \xx \in E \; .
    \end{cases}
  \end{equation}
  Suppose that $k^2$ is a Dirichlet eigenvalue for the Stokes equation
  on $\Omega$ and let $\bu$ denote the corresponding eigenfunction
  with $\bt^{-}$ the corresponding surface traction on the
  boundary $\Gamma$.
  Since $\bu$ is a Dirichlet eigenfunction, the velocity restricted
  to the boundary, $\bu^{-}$, is $0$. 
  Using the Green's theorem representation for the pair
  $(\bu,\bt^{-})$ and evaluating the surface traction
  using~\cref{lem:jump-conds}, we get
  \begin{equation}
    \bt^{-} = \left(\cDt_{k} + \frac{1}{2} \cI\right) \bt^{-} \,
    \implies \left(\cI - 2\cDt_{k} \right) 
    \bt^{-} = 0 \, .
  \end{equation}
  From the above and \cref{lem:nunull,lem:nutracli},
  we know that $\bt^{-}, \bnu \in \cN(\cI - 2\cDt_{k})$
  are two linearly independent vectors in the null space.
  Let $c = \left< \bt^{-},\bnu \right>$. Then, it
  follows that $\left<\bt^{-} -c \bnu, \bnu \right> = 0$
  and thus $\cW[\bt^{-} - c\bnu] = 0$. 
  Since $\bt^{-}$ and $\bnu$ are linearly independent, we note that
  $\bt^{-} -c\bnu \neq 0$. Combining these results, we get that
  $(\cI - 2\cDt_{k} - 2\cW) (\bt^{-} - c\bnu) = 0$. 
  Since $\bt^{-}-c\bnu$ is non-trivial, and $\cW$ is
  self-adjoint, it follows from the Fredholm alternative
  that the operator $\cI - 2\cD_{k} - 2\cW$ is also not
  invertible.
\end{proof}

The correspondence result above does not hold on multiply
connected domains.
In particular, while the operator $\cI -2\cD_{k} -2\cW$ 
is indeed not invertible when $k^2$ is a Dirichlet eigenvalue,
it turns out that the operator is also not
invertible when $k^2$ is a Neumann eigenvalue
corresponding to the interior of one of the obstacle regions,
i.e. one of the $\Omega_{i}$ with $i > 0$. The following theorem
proves this result for a region with one obstacle;
the extension to the general case is straightforward.

\begin{thrm}
  Suppose that $\Omega$ is a multiply connected domain
  given by the intersection of a domain $\Omega_{0}$
  and the exterior of a single domain $\Omega_{1}$ with
  $ \bar\Omega_1 \subset \Omega_0$.
  Then, the operator $\cI - 2\cD_{k} - 2\cW$ is not invertible
  if $k^2$ is a Neumann eigenvalue of $\Omega_{1}$.
\end{thrm}

%\begin{figure}
%\begin{center}
%\includegraphics[width=0.3\linewidth]{multiply_final}
%\end{center}
%\caption{Example of a multiply connected domain with one obstacle.}
%\label{fig:1ply}
%\end{figure}

\begin{proof}
  Suppose that $k^2$ is a Neumann eigenvalue of $\Omega_{1}$, and
  let $\tilde{\bu}$ denote the corresponding eigenfunction. 
  Note that $\tilde{\bu}$ is not identically $0$ on the boundary
  of $\Omega_{1}$, which we denote by $\Gamma_1$.
  Since $\tilde{\bu}$ is an interior Neumann eigenfunction, we note that
  the surface traction corresponding to the solution, $\bt^{-}$, is $0$
  on the boundary.
  Applying \cref{thrm:rep-theorem} to the
  solution $\tilde{\bu}$ in the interior
  and taking the interior limit we get
\begin{equation}
\tilde{\bu} = \frac{1}{2}\tilde{\bu} + \cD^{\Gamma_{1}}_{k}[\tilde{\bu}] +
\cS^{\Gamma_{1}} [\bt^{-}] \implies \frac{1}{2} \tilde{\bu} - \cD^{\Gamma_{1}}_{k}[\tilde{\bu}]
= 0\, .
\end{equation}
Note that the sign of the $\bD$ operator in the
representation theorem is switched since the normal
is pointing inwards for the boundary $\Omega_{1}$.
Thus, $\tilde{\bu}$ is a non-trivial null vector of the operator 
$\frac{1}{2}\cI - \cD^{\Gamma_{1}}_{k}$. 
Furthermore since $\tilde{\bu}$ is the boundary data of 
the solution of the oscillatory Stokes equation
in $\Omega_{1}$, we get that $\cW^{\Gamma_{1}}[\tilde{\bu}] = 0$.
Setting $\bmu = \tilde{\bu}$ on $\Gamma_{1}$, and
$\bmu = 0$ on $\Gamma_{0}$, we obtain a non-trivial null
vector for the operator $\cI - 2\cD^{\Gamma}_{k} -2\cW^{\Gamma}$
on the boundary $\Gamma = \Gamma_{0} \cup \Gamma_{1}$.
\end{proof}

The spurious eigenvalues of the operator $\cI-2\cDk-2\cW$
are demonstrated in \cref{subsec:spurannulus} on an annulus,
where both the true and spurious eigenvalues can be
determined analytically. 
Analogous with the observation in~\cite{zhao2015robust},
this lack of one-to-one correspondence between
the invertibility of the integral operator $\cI - 2\cD_{k} - 2\cW$
and the Dirichlet eigenvalues of the Stokes operator on multiply
connected domains also causes non-robustness and introduces
near-resonances for simply-connected domains which are almost
multiply-connected.
%We demonstrate this issue numerically
%in~\cref{subsec:crescent}.

\subsubsection{Dirichlet eigenvalues --- combined-field representation}
\label{subsec:mixedanalysis}

The non-robustness in using the double layer
potential representation can be
remedied by using a combined-field, or mixed layer potential,
representation, i.e. setting $\bu = (\bD_{k} + i\eta \bS_{k})
\bmu$, with $\eta$ real and positive.
Imposing the Dirichlet boundary condition
and using~\cref{lem:jump-conds}, we obtain 
\begin{equation}
  (\cI - 2\cD_{k} - 2i\eta \cS_{k}) \bmu = -2 \ff \,
  \textrm{ on } \Gamma. 
\end{equation}
As with the double layer representation, this
integral equation is rank deficient for any $k$.
Instead, we consider
\begin{equation}
(\cI - 2\cD_{k} -2i\eta \cS_{k}  -2\cW)\bmu = -2 \ff \, .
\end{equation}

We now prove that for any bounded region $\Omega$
(simply or multiply connected) with $C^{2}$ boundaries,
there exists a one-to-one correspondence between the
invertibility of the operator $\cI - 2\cD_{k}
- 2i\eta \cS_{k} - 2\cW$ and the Dirichlet eigenvalues.

\begin{thrm}
  \label{thm:cfmain}
  Suppose $\Omega$ is a bounded region defined
  by the intersection of a simply connected domain $\Omega_{0}$
  and the exteriors of a finite collection bounded simply
  connected domains $\{ \Omega_{i} \}_{i=1}^{m}$. As above,
  let $\Gamma_{i}$ denote the boundary of $\Omega_{i}$ and let
  $\Gamma = \cup_{i=0}^{m} \Gamma_{i}$ denote the boundary of
  $\Omega$. Then, the operator $\cI - 2\cD_{k} - 2i\eta \cS_{k}
  - 2\cW$ is invertible if and only if $k^2$ is not a Dirichlet
  eigenvalue for the Stokes operator on $\Omega$.
\end{thrm}

\begin{proof}
  Suppose that $k^2$ is not a Dirichlet eigenvalue for the
  Stokes equation on $\Omega$. Suppose further that $\bmu$ satisfies
  \begin{equation}
    (\cI - 2\cD_{k} -2i\eta\cS_{k} - 2\cW) \bmu = 0 \, ,
    \label{eq:mlproofrep}
  \end{equation}
i.e. $\bmu$ is in the null-space
of $(\cI - 2 \cD_{k} - 2i\eta \cS_{k} - 2 \cW)$. 
Applying the operator $\cW$ to~\cref{eq:mlproofrep} and 
using~\cref{lem:propnullspacecorr}, we get
\begin{equation}
0 = \cW [(\cI - 2\cD_{k} - 2i\eta\cS_{k} - 2\cW)\bmu] = -2\cW [\bmu] \, .
\end{equation}
Thus~\cref{eq:mlproofrep} reduces to
\begin{equation}
(\cI - 2 \cD_{k} - 2i\eta \cS_{k})\bmu = 0\, .
\end{equation}
Suppose now $\bu = -2\bD_{k}[\bmu] -2i\eta\bS_{k}[\bmu]$
in $\Omega$. Then $\bu$ is a solution to the oscillatory
Stokes equation in $\Omega$, and applying~\cref{lem:jump-conds},
we get that the interior limit of the velocity
$\bu^{-} = (\cI - 2\cD_{k} -2i\eta\cS_{k}) \bmu = 0$ on $\Gamma$.
Since $k^2$ is not a Dirichlet eigenvalue for the Stokes equation
on $\Omega$, we conclude that $\bu \equiv 0$ in $\Omega$. 
This in particular implies that the interior limit of the
surface traction, denoted by $\bt^{-}$, is $0$ on $\Gamma$.

Using~\cref{lem:jump-conds} we observe that $\bt^{+}
= 2i\eta \bmu(\xx)$ and $\bu^{+} = -2\bmu(\xx)$ on $\Gamma$, i.e.
$\bt^{+}+i\eta \bu^{+}=0$ and $\bu^{+}$ satisfies the homogeneous
exterior impedance problem.
We first show that $\bmu=0$ on $\Gamma_{0}$. 
To this end, note that $\bu$ is a radiating solution
of the oscillatory Stokes equation in the exterior $E$,
since $\cW[\bmu] = 0$ implies that $\int_{\Gamma} \bmu
\cdot \bnu = 0$. From the uniqueness of the impedance problem
in the exterior $E_0$ of $\Omega_0$, we conclude that
$\bu \equiv 0$ in $E_0$ as well, which in particular
implies that $\bu^{+}=0$ on $\Gamma_{0}$.   
Using the jump conditions in~\cref{lem:jump-conds}
again, we get that
$2\bmu = \bu^{-} - \bu^{+} = 0$ on
$\Gamma_{0}$. 

\begin{remark}
  Note that there is potential for confusion here
  in that the exterior limit with respect to $\Omega$ 
  for the boundary $\Gamma_{j}$ is the traditional interior 
  limit with respect to the obstacle region $\Omega_{j}$.
\end{remark}

To show that $\bmu=0$ on
$\Gamma_{j}$, we observe that $\bu$ is also a
solution to the oscillatory Stokes equation in each
of the obstacles $\Omega_{j}$.
Using the jump conditions in~\cref{lem:jump-conds},
we get that $\bt^{+} = 2i\eta \bmu$ and $\bu^{+} = -2\bmu$.
However, the normal is inward pointing inside $\Omega_{j}$
on the boundary $\Gamma_{j}$. 
If we revert back to the normal being defined as 
an outward normal to $\Omega_{j}$, then
the boundary conditions on $\Gamma_{j}$ 
is $\bt - i \eta \bu = 0$. 
From the uniqueness of solutions to the interior 
impedance problem, we conclude that $\bu \equiv 0$
in $\Omega_{j}$, which in particular implies
that $2\bmu = \bu^{-}-\bu^{+} = 0$ for $\xx \in \Gamma_{j}$,
$j=1,2,\ldots m$.
Thus, $\cI - 2\cD_{k} -2i\eta\cS_{k} -2\cW$ is
invertible when $k^2$ is not a Dirichlet eigenvalue
for the Stokes equation on $\Omega$.

From \cref{thrm:rep-theorem}, we have 
\begin{equation} 
  \bS [\bt](\xx) - \bD[\bu](\xx) = \begin{cases} 
    \bu(\xx) &\quad \xx \in \Omega \, , \\
    0 &\quad \xx \in \R^2 \setminus \bar{\Omega} \; .
    \end{cases}
  \end{equation}
Suppose that $k^2$ is Dirichlet eigenvalue for
the Stokes equation on $\Omega$ and let $\bu$
denote the corresponding eigenfunction and $\bt$ denote
its surface traction. Note that \cref{thrm:rep-theorem}
implies that $\cS_{k}[\bt^{-}] = 0$, since
$\bu^{-}=0$ on $\Gamma$. Applying \cref{thrm:rep-theorem}
to the pair $\bt^{-}, \bu^{-}$ and evaluating the
traction on $\Gamma$ using~\cref{lem:jump-conds},
we get
\begin{equation}
\bt^{-} = (\cDt_{k} + \frac{1}{2} \cI) \bt^{-} \, . 
\end{equation}
Combining these two identities, we get
that
\begin{equation}
(\cI - 2\cDkt -2i\eta \cS_{k}) \bt^{-} = 0 \, .
\end{equation}
As in the proof of~\cref{thm:dlmain}, letting
$c = \left< \bt^{-} ,\bnu \right>$, it follows that
\begin{equation}
  (\cI - 2\cDkt - 2i \eta \cS_{k} - 2\cW)
  (\bt^{-} - c\bnu) = 0 \, ,
\end{equation}
where $\bt^{-} - c\bnu \neq 0$.
Since $\bt^{-} - c\bnu$ is non-trivial and both
$i\eta \cS_{k}$ and $\cW$ are self-adjoint with respect
to the bilinear form \cref{eq:bi_form},
it follows from the Fredholm alternative
that the operator $\cI - 2\cDk -2i\eta \cSk - 2\cW$
is also not invertible.
\end{proof}

\subsection{Fredholm determinants}
\label{sec:dets}
In this section, we show
how the Fredholm determinant can be used
as a computational tool for detecting the
non-invertibility of $\cI - 2\cD_{k} - 2\cW$.
The arguments here follow the structure of the
analogous arguments in~\cite{zhao2015robust}
for Laplace eigenvalues.

Let $\cJ_{1}(X)$ denote the space of trace class operators 
on $X$, where $X$ is a Hilbert space, which is a
subspace of the space of compact operators on $X$.
A compact operator $\cA$ with eigenvalues
$\lambda_{i}, i\in \mathbb{N}$ is in $\cJ_{1}(X)$ if
$\sum_{i} |\lambda_{i}| < \infty$.
If $\cA$, is a trace class operator, then  
the Fredholm determinant of the operator $\cI + \cA$
is defined by
\begin{equation}
\text{det}(\cI +\cA) = \prod_{i=1}^{\infty} (1+\lambda_{i}) \, .
\end{equation}

So far, we have discussed the Fredholm theory in the space
$C(\Gamma)\times C(\Gamma)$ equipped with the bilinear form~\cref{eq:bi_form}.
However, it is more convenient to discuss the theory of Fredholm
determinants on Hilbert spaces. 
We note that both the operators $\cD_{k}$ and $\cS_{k}$ 
are also compact operators mapping $Y \to Y$
where $Y = \mathbb{L}^{2}(\Gamma) \times \mathbb{L}^{2}(\Gamma)$.
Furthermore, it is well-known that the spectrum of compact
operators with weakly singular kernels coincide on 
$C(\Gamma)\times C(\Gamma)$ and $Y$ (see~\cite{kress1989linear},
for example).
So for the rest of the section, we present the discussion of 
the relevant operators on $Y$ instead of $C(\Gamma)\times C(\Gamma)$.

%It follows from standard results in complex
%analysis, that the Fredholm determinant is finite and well-defined if 
%$\text{det}(\cI+A) < \infty$ if $\sum_{i} |\lambda_{i}| < \infty$,
%i.e. if $A$ is in trace class.

The operator $-2\cD_{k} - 2\cW$ is trace class:
\begin{lem}
  Suppose that $\Gamma$ is a $C^2$ curve.
  Then $-2\cD_{k} - 2\cW \in \cJ_1(Y)$
  for all $k \in \mathbb{C} \setminus \{0\}$ 
\end{lem}
\begin{proof}
Using Bessel function asymptotics, we note that the
kernel of $\cD_{k}$ given by $\TT_{\cdot,\cdot,\ell}\nu_{\ell}(\bx,\by)$
has a leading order singularity of
$|\bx-\by|^2 \log{|\bx-\by|^2}$ as $\bx\to \by$ 
for all $k \in \mathbb{C} \setminus \{ 0\}$.
It follows from the criteria listed
in~\cite[Sec. 2]{bornemann2010numerical} that $\cD_{k}$
is a trace-class operator.
Since $\cW$ is a rank-one perturbation independent of $k$,
and trace-class operators are a vector space, we conclude
that $-2\cD_{k} - 2\cW$ is also a trace-class operator.
\end{proof}

Let $f(k) = \text{det}(\cI - 2\cD_{k} - 2\cW)$.
First, note that $f(k)$ is an analytic function of $k$
for $k \in \mathbb{C} \setminus \{0 \}$, since the kernel
of $\cD_{k}$ is an analytic function of $k$ on that domain, 
and the Fredholm determinant of an analytic operator 
is analytic on the domain of analyticity
of the operator (see~\cite{zhao2015robust}, for example).

The zeros of the Fredholm determinant indicate when the
operator is not invertible.
The following lemma summarizes this result.
\begin{lem} \label{lem:detzeros}
  With $f(k)$ defined as above, $f(k) = 0$ if and only if
  $\cI - 2\cD_{k} -2\cW$ is not invertible.
\end{lem}
\begin{proof}
  The proof is standard; see, for example,
  \cite[p. 34]{simon2005trace}.
\end{proof}

When $\Omega$ is simply connected, \cref{lem:detzeros} and 
\cref{thm:dlmain} together imply that $f(k) = 0$
if and only if $k^2$ 
is a Dirichlet eigenvalue of the Stokes equation.
This reduces the problem of finding eigenvalues to
finding the roots of an analytic function.

We now show how this fact can be used to numerically
estimate the Dirichlet eigenvalues.
Suppose that $D_{k}^{N}$ is a Nystr\"{o}m discretization 
of the operator $-2\cD_{k} - 2\cW$ when the boundary 
$\Gamma$ is discretized with $N$ points. 
Let $f^{N}(k) = \text{det}(I + D_{k}^{N})$
where here $\text{det}$ is the standard matrix determinant.
Note that the discretized matrix also depends on the choice
of quadrature rule used in the Nystr\"{o}m discretization
of the operator.

In~\cite{zhao2015robust}, the authors prove that
for computing the Laplace eigenvalues on regions with 
analytic boundaries, when the integral operators are
discretized using Kress quadrature --- a spectrally accurate
quadrature rule for such kernels, see~\cite{kress1991boundary} ---
the determinant of the Nystr\"{o}m discretized operators
at the true eigenvalues converge to $0$ exponentially
in $N$.
Thus, if the eigenvalues have multiplicity $1$,
i.e. the derivative of the determinant is non-zero
at the true-eigenvalues, then the analyticity of the
discretized determinant implies that the zeros
of the determinant of the Nystr\"{o}m discretization
of the linear operator converge exponentially to
the true Dirichlet eigenvalues for Laplace's equation.

The proof presented in~\cite{zhao2015robust} applies
to the BIE approach for computing the Dirichlet eigenvalues
of the Stokes operator as well.
The result is summarized below.
\begin{thrm}
\label{thm:mainconvfreddet}
Suppose that $\Omega$ is a simply connected
domain with an analytic boundary. Let $k_{j}^2$, $j=1,2,\ldots M$
denote all the Dirichlet eigenvalues of Stokes equation
on $\Omega$ contained in the interval $[a,b]$. 
Suppose further that all the eigenvalues have multiplicity $1$.
Let $f^{N}(k) = \text{det}(I+D^{N}_{k})$, where $D^{N}_{k}$ 
is the Nystr\"{o}m discretization of $-2\cD_{k} - 2\cW$ with Kress
quadrature.
Then there exists $N_{0} \in \mathbb{N}$ such
that for all $N>N_{0}$, 
$f^{N}(k)$ has exactly $M$ zeros on the interval $[a,b]$.
Let $\omega_{j}$, $j=1,2\ldots M$ denote the zeros of $f^{N}$.
Furthermore, there exist constants $a>0$ and $C$, 
such that $\sup_{j=1}^{M} |\omega_{j} - k_{j}| < C e^{-aN}$.
\end{thrm}

\begin{proof}
  The proof follows from small modifications of the
  proofs contained in~\cite{zhao2015robust}. 
\end{proof}

\begin{remark}
In practice, using Kress quadrature for large
problems is problematic owing to the global
nature of the quadrature rule.
First, the use of a global rule does not allow
for adaptive refinement at a complicated, local
feature of the boundary.
Second, the integration weight in each entry
of the matrix depends on both the column and the row
in a non-separable way.
As a result, the fast multipole method is not
directly applicable to the resulting matrix
and many fast-direct methods for computing
the determinant lose efficiency (for
instance, the reasoning behind the use of a
{\em proxy surface} \cite{cheng2005compression}
no longer holds).
Over the last two decades, many high-order quadrature
methods which are compatible with
the fast multipole method and fast-direct methods
have been developed.
Our numerical experiments, see \cref{subsec:convannulus},
suggest that the zeros of the 
determinants of linear systems discretized using 
these quadrature methods are also high order approximations
of Dirichlet eigenvalues for the Stokes operator ---
the error is observed to be proportional to the quadrature error
for the eigenfunction $\bt^{-}$ associated with the eigenvalue.
We leave a proof of this to future work.
\end{remark}

\begin{remark}
The same analysis does not carry through for the operator
$\cI - 2 \cD_{k} - 2i\eta \cS_{k} - 2\cW$, 
since $\cS_{k}$ is not a trace class operator.
For brevity, let $\cC_{k} = -2\cD_{k} - 2i\eta \cS_{k} - 2\cW$.
The operator $\cC_{k}$ is 
in $\cJ_{2}(Y)$ where
$\cJ_{2}(Y)$ is the space of Hilbert-Schmidt operators
on $Y$ (the singular values of the operator are square
summable, as opposed to being summable).
Thus the Fredholm determinant of $\cI + \cC_{k}$
is not necessarily finite. 
However, as noted in~\cite{zhao2015robust}, the convergence
result~\cref{thm:mainconvfreddet}
should be true up to a logarithmic factor in the rate
of convergence, since the singular values of the operator
$\cC_{k}$ decay like $\frac{1}{n}$, and the
Fredholm determinant diverges logarithmically. 
In~\cref{subsec:convannulus}, we demonstrate this fact
numerically on the annulus, where the eigenvalues are
analytically known. 
\end{remark}

% flatex input end: [03analysis.tex]

%
% flatex input: [04numerical.tex]
\section{Numerical Results}
\label{sec:numerical}

In this section, we demonstrate the analytical
claims above with numerical examples and
highlight the performance of the BIE approach
with demonstrations on domains of analytical
and practical interest. The software used to
generate the figures is available online
\footnote{\texttt{https://doi.org/10.5281/zenodo.2641296}}.

\subsection{Numerical methods}

First, we describe the numerical tools needed
to compute Stokes eigenvalues in a BIE
framework.

\subsubsection{Discretizing the BIE}

In order to turn the BIEs analyzed
above into discrete linear systems,
we require some standard techniques
from the BIE literature.

Let the boundary be divided into $N_p$
panels.
We parameterize panel $j$ as
$\bx_j(t)$, with $t$ ranging over the
interval $[-1,1]$.
Each component of $\bx_j$ is taken to be
a polynomial interpolant over the
standard 16th-order Legendre nodes on
$[-1,1]$, denoted by $t_n$, so that
the total number of discretization
points is $N=16N_p$.
See \cref{fig:panels} for an example
discretization.

\begin{figure}
\centering
\includegraphics[width=0.5\textwidth]{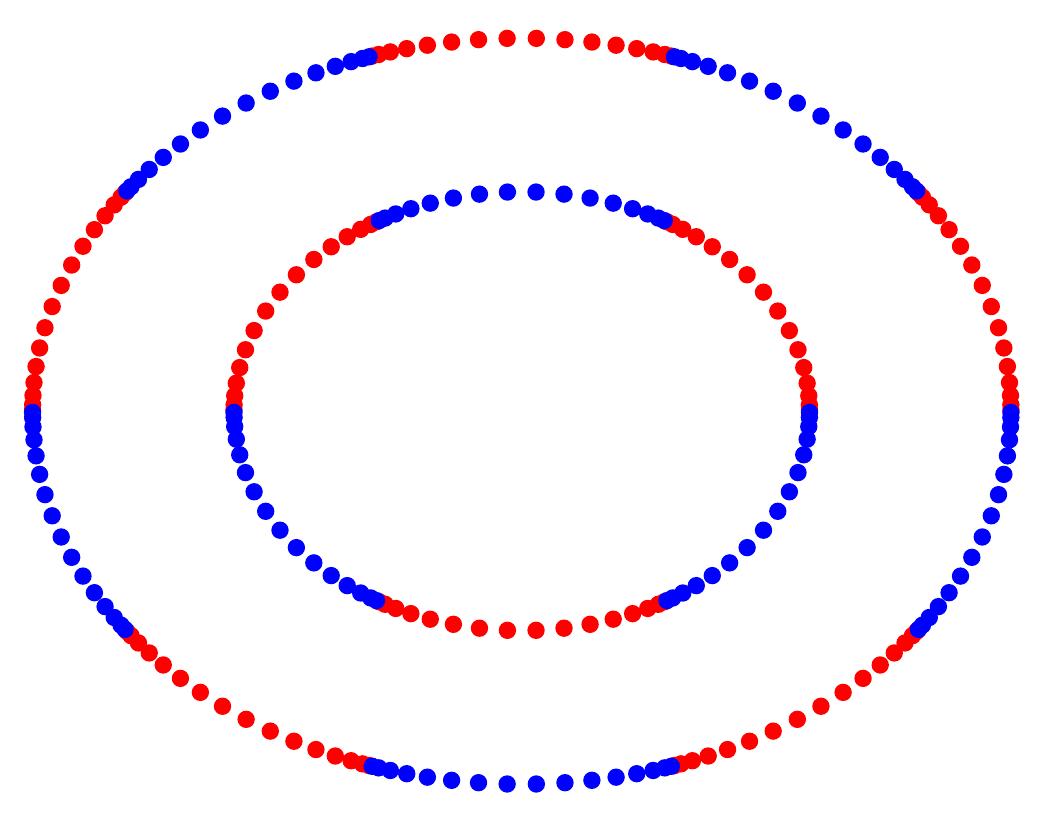}
\caption{Sample discretization of an annulus, where the inner
circle, $r=1$ is discretized using $6$ panels and the outer circle $r=1.7$
is discretized using $10$ panels.}
\label{fig:panels}
\end{figure}
Another important quantity below is the
arc-length density of a panel, which
we denote by $s_j(t) := |\bx'_j(t)|$.
Finally, we denote the set of panels
which are adjacent to panel $j$
by $A(j)$. On a closed curve,
$A(j)$ contains two integers.

The integral kernels of the single and
double layer potentials have weak
singularities of the form $|\xx-\yy|^p\log |\xx-\yy|$
for some $p \in \N_0$, which require special
quadrature rules to achieve high-order accuracy.
In the examples below, we use generalized
Gaussian quadrature (GGQ)~\cite{bremer2010}.
To demonstrate the idea, we consider
evaluating the convolution of a kernel
$K$ with a density $\sigma$
at the boundary node $\xx_j(t_l)$.
GGQ is a Nystr\"{o}m-type discretization ---
the density is approximated 
by its values at the discretization nodes,
which we denote by
$\sigma_{qp} := \sigma(\xx_q(t_p))$.
The basis of a GGQ rule is a set of 
support nodes and weights for the
contribution to the integral from the
``self'' panel (panel $j$) and the adjacent
panels (with index in $A(j)$).

For the self panel, there is a special set
of nodes and weights for each interpolation
point. Denote the nodes and weights
for interpolation point $l$ by $t^{(l)}_{n}$
and $w^{(l)}_{n}$, respectively, with
$1\leq l \leq 16$ and $1\leq n \leq N_s$.
The adjacent panels are handled by a single
set of over-sampled support nodes and weights.
We denote these nodes and weights
by $\tilde{t}_n$ and $\tilde{w}_n$, respectively,
for $1 \leq n \leq N_a$.
For the 20th-order rule we used, $N_s = 16$ and
$N_a = 48$. 
The contribution of other panels is assumed
to be given to high accuracy by the standard
Gauss-Legendre weights, which we denote
by $w_n$.
Adding these contributions together, we obtain the
quadrature

\begin{multline}
  \int_\Gamma K(\xx_i(t_l),\yy) \, \sigma(\yy)
  \, dS(\yy) \approx \\
  \sum_{p=1}^{16} \sum_{n=1}^{N_s}
  w^{(l)}_{n} K(\xx_i(t_l),\xx_i({t}^{(l)}_{n}))
  s_i({t}^{(l)}_{n}) B^{(l)}_{np} \sigma_{ip} \quad \textrm{(self)}
  \\
  + \sum_{q\in A(i)} \sum_{p=1}^{16} \sum_{n=1}^{N_a}
  \tilde{w}_jK(\xx_i(t_l),\xx_q(\tilde{t}_n)) s_q(\tilde{t}_n)
  C_{np} \sigma_{qp}
  \quad \textrm{(adjacent)} \\
  + \sum_{q\neq i, q\not\in A(i)} \sum_{p=1}^{16}
  w_p K(\xx_i(t_l),\xx_q(t_p)) s_q(t_p)
  \sigma_{qp} \quad \textrm{(far)} \nonumber \; ,
\end{multline}
where $\bB^{(l)}$ and $\bC$ are interpolation
matrices from the standard Legendre nodes
to the self and adjacent panel support nodes,
respectively. Observe that the quadrature is
linear in $\sigma_{qp}$. In practice, we pre-compute
and store the self and adjacent matrix entries for each
interpolation point, which is a parallelizable
$O(N)$ calculation. The ``far'' interactions
are computed on-the-fly.

\begin{remark}
  \label{rmk:levelrestrict}
  We ensure that ``far'' interactions
  are handled to high precision by requiring that
  no two adjacent panels differ in length
  by more than a factor of 2. On a domain which does
  not nearly self-intersect this
  guarantees that no ``far'' interactions occur
  which are much closer than 1/2 of a panel away
  (assuming panels are relatively flat).
  Because the location of the singularity is
  bounded away from the panel and the smooth
  rule is of high order, we obtain a quadrature
  rule with sufficient precision.

  The overall order of accuracy of the GGQ we use
  is 20th-order, up to the precision of the ``far''
  interactions.
\end{remark}

\subsubsection{Fast determinant method}

Once the discretization is set, we can form
a compressed representation of the system matrix
using recursive skeletonization~\cite{ho2012fast}.
We use the implementation of this procedure
included in the fast linear algebra in
MATLAB (\texttt{FLAM}) package
\cite{hoFLAM_1253582}.
At low-to-medium frequencies, the scaling
of the recursive skeletonization algorithm
is $O(N\log N)$ in operation count and
storage and, by using a generalization
of the Sylvester determinant formula,
allows for a fast determinant
calculation in $O(N\log N)$ time as a
follow-up step.
At higher-frequencies,
the recursive skeletonization procedure,
which is based on the assumption that off-diagonal
blocks of the matrix are of low rank,
breaks down and does not offer a speed advantage.
These algorithms take a precision parameter
$\epsflam$ which determines the
accuracy to which any sub-blocks of the matrix
should be compressed. In all experiments,
we set $\epsflam = 10^{-14}$.

The compressed representation also allows
for fast applications of the system matrix,
its transpose, the inverse of the system
matrix, and the inverse transpose to
vectors.
In particular, this allows us to estimate the
smallest singular values by performing
randomized subspace iteration, see
\cite[Algorithm 4.4]{halko2011finding},
on the inverse operator.
Below, we use the smallest singular value
as a measure of the quality of the
eigenvalues found by approximating the
roots of the determinant.
We also evaluate the second smallest singular
value if the root finding procedure suggests a
possible double root. 

\subsubsection{Interpolation and root-finding}

To estimate the eigenvalues, we fit a Chebyshev
interpolant to the discretized determinant as a
function of $k$ on intervals.
This is done adaptively so that the Chebyshev
coefficients of the determinant have decayed
to the point that the ratio of the last
coefficient to the largest coefficient is below
some threshold.
In all experiments, we set this threshold
as $\epscheb = 10^{-13}$.
We perform this fit using the \texttt{chebfun}
utility in the package of the same name
\cite{driscoll2014chebfun}
so that we can make use
of the \texttt{roots} utility to approximate
the roots of the determinant.

The \texttt{roots} utility returns the roots
of the polynomial in the complex plane, with
some minimal internal processing to remove
spurious roots.
Because our numerical determinant evaluation
is somewhat noisy and we fit the function up to
precision $\epscheb$, we perform some
further post-processing to eliminate remaining
spurious roots.
Let $k^{(l)}_\cheb$ denote the roots of the interpolants.
We ignore any of the returned roots with $|\imag(k^{(l)}_\cheb)|
> \sqrt{\epscheb}$, as these are too far from real-valued
to be non-spurious.
For the remaining roots, we consider the
properties of $\real(k^{(l)}_\cheb)$.
We inspect any pairs of roots $(k^{(p)}_\cheb,k^{(q)}_\cheb)$
for which $|\real(k^{(p)}_\cheb-k^{(q)}_\cheb)| < \sqrt{\epscheb}$,
as these are possibly spurious double roots.
For these pairs, we compare the right singular
vector of the appropriate BIE operator corresponding
to the smallest singular value for each of $\real(k^{(p)}_\cheb)$
and $\real(k^{(q)}_\cheb)$, which we denote by
$\bv_p$ and $\bv_q$.
If $\|\bv_p - \bv_q \bv_q^* \bv_p \| < 10^{-5}$, then
we consider the pair to be spurious.
For these near double roots, we also check
that there is no two dimensional null-space
corresponding to the root by estimating the
second smallest singular value of the BIE operator.
If this is larger than $10^{-5}$, then we declare
it to be a simple root.

We can obtain an a posteriori estimate of the
error in a computed root as follows.
Let $f$ denote an analytic function,
$P$ be the polynomial interpolant
of that function over some interval,
$\delta f = f-P$ be the difference,
and $k_\cheb$ denote a computed root
of $P$ which is simple (i.e. assume
that $P'(k_\cheb) \ne 0$).
The algorithm used by \texttt{chebfun}
to approximate the roots of $P$ is
backward stable~\cite{noferini2017chebyshev}.
Therefore the error in the roots will be
small relative to the error of the fit and
we set $P(k_\cheb) = 0$ below.
Suppose that
$f(k_\cheb + \delta k) = 0$ for some
small $\delta k$. Then

\begin{align*}
  0 &= f(k_\cheb + \delta k) \\
  0 &= P(k_\cheb + \delta k) + \delta f(k_\cheb + \delta k) \\
  \delta k &= -\frac{\delta f(k_\cheb + \delta k)}{P'(k_\cheb)} + O(\delta k^2) .
\end{align*}
In practice, we can obtain an approximate upper
bound for $|\delta f(k_\cheb+\delta k)|$
as $\epscheb \|P\|_\infty$ so that $\epscheb \|P\|_\infty/
|P'(k_\cheb)|$ provides an approximate upper bound
for the error in the root.

\subsection{Eigenvalues of an annulus}
We test our numerical machinery and validate our analytical 
and numerical claims
by comparing the results to the true eigenvalues on the annulus
which are known analytically (see~\cref{sec:annul_dir_exact}).
In all of the examples below, we work on the annulus $r_{1}<r<r_{2}$
with $r_{1} = 1$ and $r_{2} = 1.7$.
If the inner boundary is discretized using $N_{1}$ panels, 
then the outer boundary is discretized using $N_{2} = 
\lceil r_{2}/r_{1} N_{1} \rceil +1$ panels to ensure that the
panels are approximately the same length on both the boundaries. 
The total number of discretization points 
is then given by $N = 16(N_{1} + N_{2})$.
Let $D^{N}_{k}$ denote the linear system corresponding
to the Nystr\"{o}m discretization of
$-2\cD_{k} -2\cW$ using generalized Gaussian quadrature, 
and let $C^{N}_{k}$ denote the linear system corresponding
to the Nystr\"{o}m discretization of 
$-2\cD_{k} - 2i\cS_{k} -2\cW$.
Let $f_{D}^{N}(k) = \text{det}(I+D^{N}_{k})$, and 
$f_{C}^{N}(k) = \text{det}(I+C^{N}_{k})$.

\subsubsection{Convergence study}
\label{subsec:convannulus}
We demonstrate that 
for sufficiently large $N$, if $k_{0}$ is a Dirichlet 
eigenvalue of the annulus, 
then $f_{D}^{N}(k_{D}) = 0$ and $f_{C}^{N}(k_{C}) = 0$
where $|k_{D} - k_{0}| \lesssim N^{-20}$, 
and $|k_{C} - k_{0}| \lesssim N^{-20}$.
Recall that the GGQ we use has an expected order
of convergence of $N^{-20}$ for evaluating convolutions
with such integral kernels,
so that the error of the roots is observed to have the
same order as the quadrature rule.
In~\cref{fig:conv}, we show this result for $k_{0} = 13.48025717955055$
and plot the errors $|k_{D}-k_{0}|$ and $|k_{C}-k_{0}|$ as a function
of $N$. 

\begin{figure}
  \centering
  \includegraphics[width=0.98\textwidth]{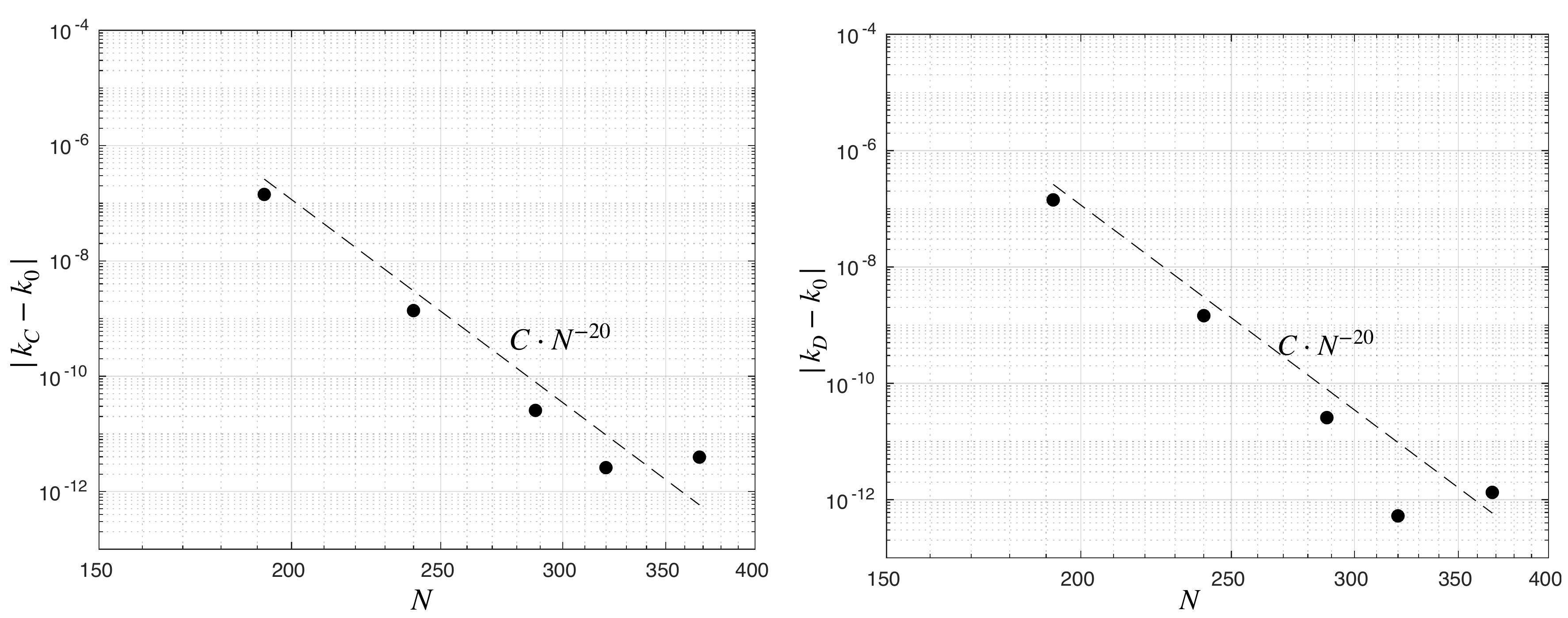}
  \caption{Convergence studies for the Dirichlet eigenvalues
  computed using the integral equations 
  $\cI - 2\cD_{k} - 2\cW$ (left) and
  $\cI - 2\cD_{k} - 2i \cS_{k} - 2\cW$ (right).}
  \label{fig:conv}
\end{figure}

\subsubsection{Spurious eigenvalues}
\label{subsec:spurannulus}
As noted in~\cref{subsec:dlanalysis},
if $k_{0}$ is a Neumann eigenvalue corresponding to the interior
inclusion, which in our case is the disk $r\leq r_{1}$, then
$f_{D}^{N}(k_{D}) = 0$ with $|k_{D}-k_{0}| = O(\varepsilon)$, even though $k_{0}$ 
is not a Dirichlet eigenvalue of the annulus, i.e.
the integral equation $-\cI - 2\cD_{k} -2\cW$ has a spurious eigenvalue.
In~\cref{fig:spur}, we demonstrate this result and also show that 
$f_{C}^{N}(k_{0}) \neq 0$, i.e., the combined field representation is
robust and invertible at all values of $k$ which are not the Dirichlet
eigenvalues of the annulus.

\begin{figure}
\centering
\includegraphics[width=0.98\linewidth]{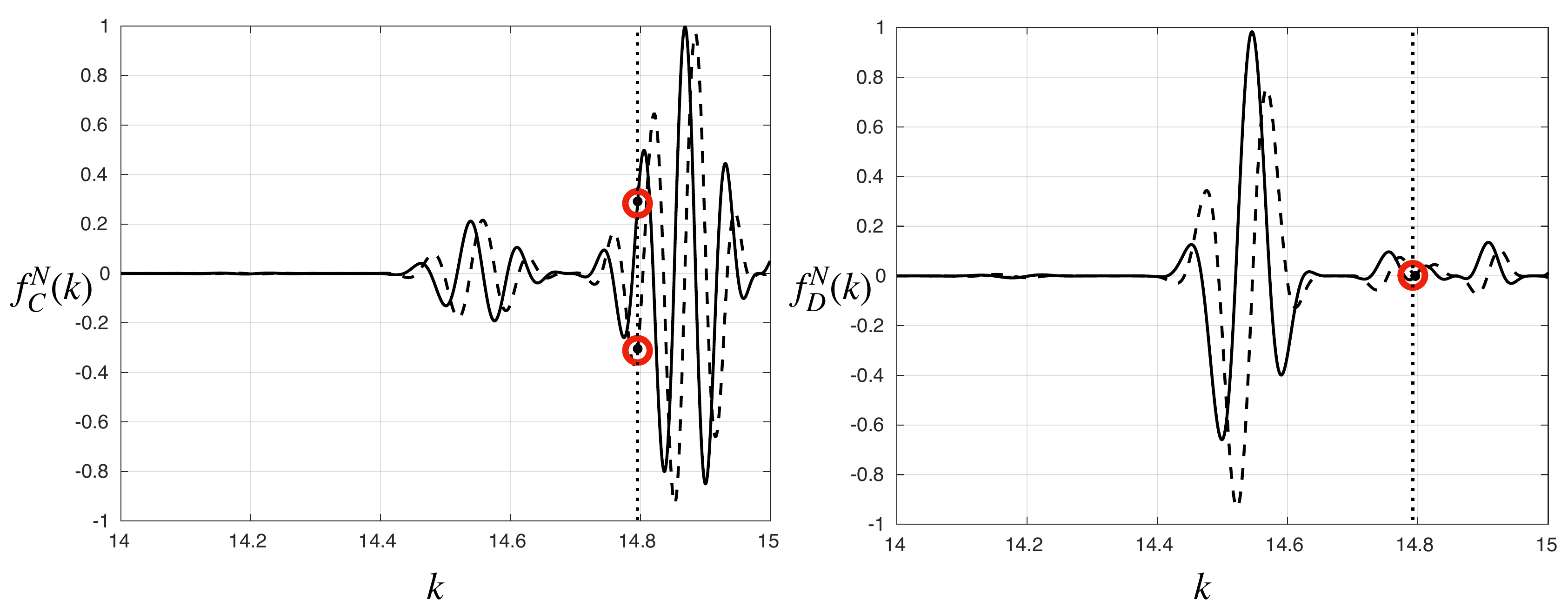}
\caption{The values of the discretized determinants $f_{C}^{N}(k)$ (left)
and $f_{D}^{N}(k)$ on the interval $k=[13,14]$ with $N=368$. The vertical
dotted line denotes the spurious eigenvalue $k_{0} = 14.79595178235126$ 
and $f_{D}^{N}(k_{D}) = 0$ with $|k_{D}-k_{0}|=6.8\times 10^{-12}$.
On the other hand $|f_{C}^{N}(k_{0})|=0.42$, and thus $\cI-2\cD_{k}-2i\cS_{k} 
-2\cW$ has no 
spurious eigenvalue in the neighborhood of $k=k_{0}$.}
\label{fig:spur}
\end{figure}

\subsubsection{Speed}
\label{subsec:speed}
In this section, we demonstrate the $O(N\log{N})$ scaling of evaluating
$f^{N}_{C}(k)$ as long as $N$ is large enough to resolve the interactions
at the Helmholtz parameter $k$. 
When $N$ is smaller than that, we observe a worse scaling since the
assumption that far-interactions are low-rank is no longer valid at
the tolerance of FLAM.
We plot the timing results corresponding to three different values of 
$k$ in~\cref{fig:speed}. The times are as recorded for a laptop with
16Gb of RAM and an Intel Core i7-6600U CPU at 2.60GHz with 4 cores.
\begin{figure}
\centering
\includegraphics[width=0.5\textwidth]{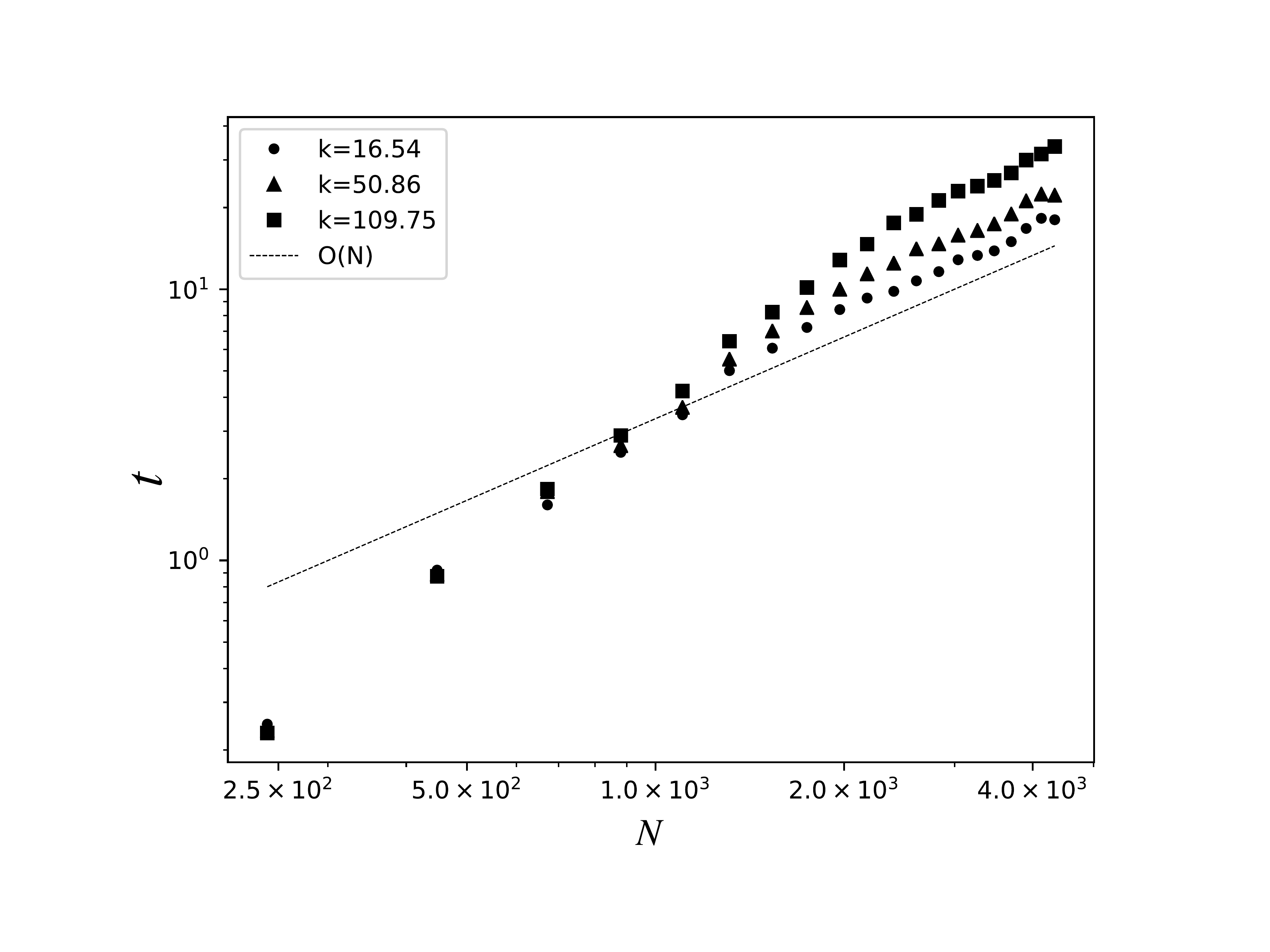}
\caption{Time taken $(t)$ in seconds to evaluate the determinant $f_{C}^{N}(k)$ as
a function of $N$ for three different values of $k$.}
\label{fig:speed}
\end{figure}

\subsection{Eigenvalues of a barbell-shaped domain}
\label{subsec:barbell}

\begin{figure}
  \centering
  \begin{subfigure}[t]{0.4\textwidth}
    \centering
    \includegraphics[width=\textwidth]{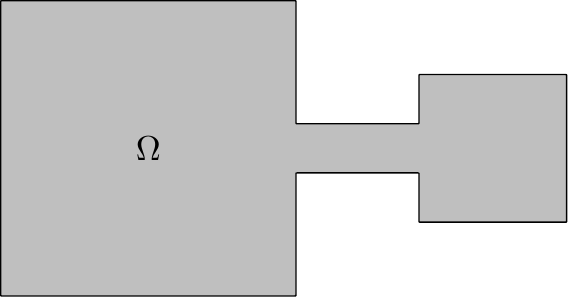}
    \caption{A barbell-shaped domain.}
    \label{subfig:barbell_bdry}
  \end{subfigure}
  ~
  \begin{subfigure}[t]{0.4\textwidth}
    \centering
    \includegraphics[width=\textwidth]{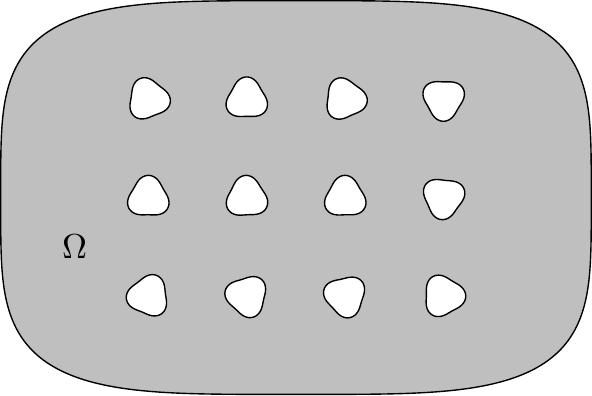}
    \caption{A domain with several inclusions.}
    \label{subfig:many_inclusions_bdry}
  \end{subfigure}
  \caption{Computational domains.}
\end{figure}

\begin{figure}
  \centering
  \includegraphics[width=\textwidth]{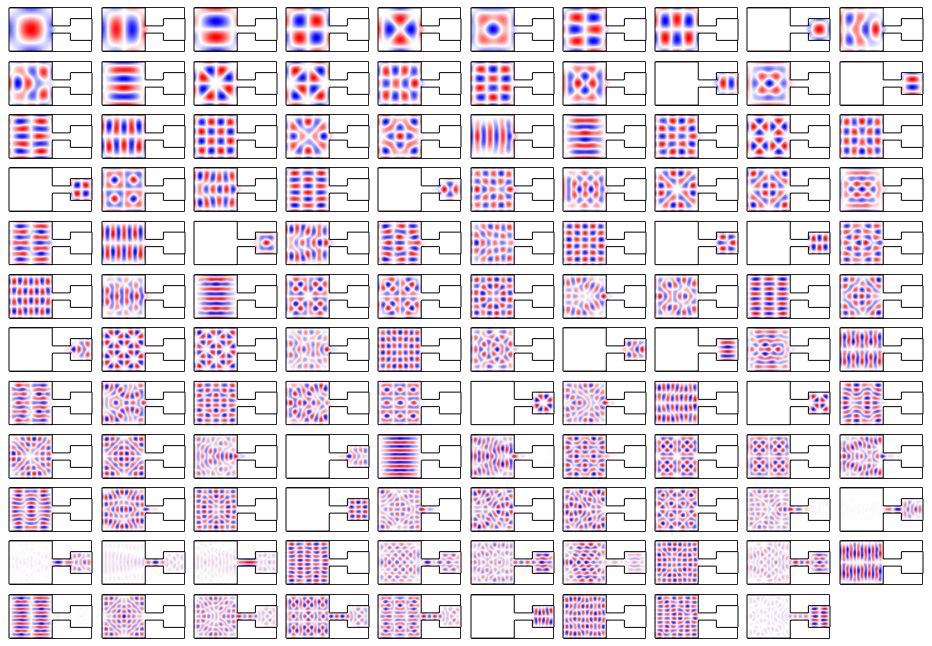}
  \caption{Vorticity plots of the first 119 eigenfunctions
    of the barbell-shaped domain.}
  \label{fig:barbell_gallery}
\end{figure}

\begin{figure}
  \centering
  \begin{subfigure}[t]{0.4\textwidth}
    \centering
    \includegraphics[width=\textwidth]{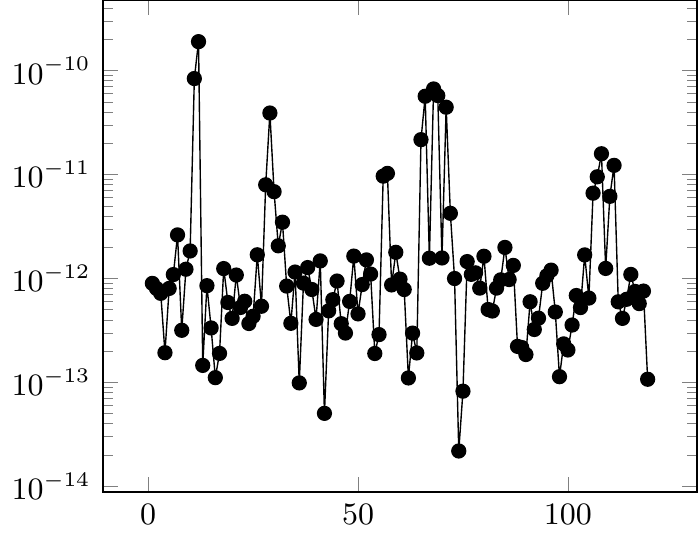}
    \caption{Smallest singular value of the BIE operator.}
    \label{subfig:barbell_sings}
  \end{subfigure}
  ~
  \begin{subfigure}[t]{0.4\textwidth}
    \centering
    \includegraphics[width=\textwidth]{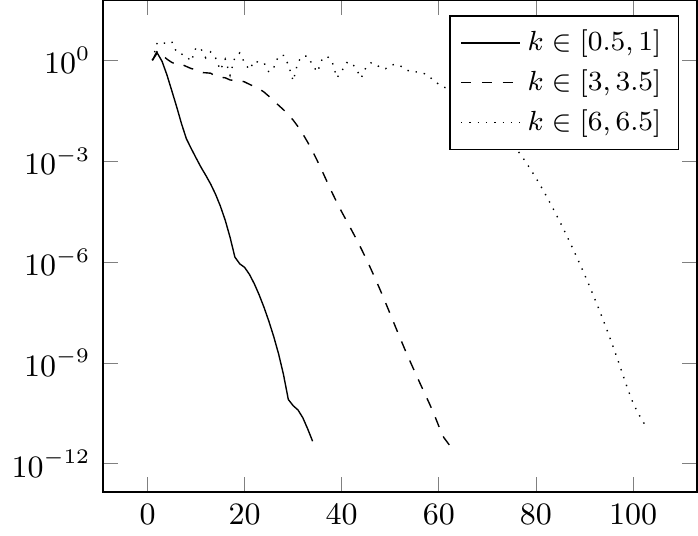}
    \caption{Normalized Chebyshev coefficients of the
      determinant on 3 intervals in $k$.}
    \label{subfig:barbell_coeffs}
  \end{subfigure}
  \caption{Diagnostics for the first 119 barbell eigenvalues.}
  \label{fig:barbell_diagnostics}
\end{figure}

We consider the barbell-shaped domain in \cref{subfig:barbell_bdry}.
This domain is the union of a square of side-length 6,
a square of side-length 3, and a ``bridge'' connecting
them of height 1 and width 5/2.
For the sake of simplicity, we round the corners of the domain
to obtain a smooth object.
Applying the approach described in~\cite{epstein2016smoothed},
the corners of the domain are rounded by convolving with
the Gaussian kernel
\begin{equation}
  \nonumber
\phi(x) = \frac{1}{\sqrt{2\pi h}} e^{-x^2/(2 h^2)} \, ,
\end{equation}
with $h\approx 0.06$. This leaves the domain unperturbed
to high precision outside of a radius of $0.1$ around
each corner.
The eigenfunctions of such a domain display the well-known
localization property~\cite{trefethen2006computed}:
many of the eigenfunctions are approximately supported
within one of the squares.
We compute these eigenfunctions corresponding to
eigenvalues $k^2$ with $k$ in the range
$0.5 \leq k \leq 6.5$.

The panels are divided adaptively so that the smallest
panels in the rounded corners are smaller than $10^{-2}$,
which keeps the panels relatively flat.
This results in $N_p = 412$ after enforcing the
level-restriction property described in
\cref{rmk:levelrestrict}
and enforcing that no panel is larger than
one wavelength for the largest $k$
(here $\lambda=2\pi/6.5$).

As this is a simply-connected domain,
the eigenvalues are estimated by finding the values
$k$ for which $\cI-2\cDk-2\cW$ is non-invertible.
Let $f^N(k) = \det (\cI^N-2\cDk^N-2\cW^N)$.
To find the roots of $f^N(k)$, we fit a \texttt{chebfun}
representation of $f^N(k)$ on each of the intervals
$[j/2,(j+1)/2]$ for $j = 1,\ldots,12$.
We plot the absolute value of the Chebyshev coefficients
(normalized by the absolute value of the first coefficient)
of $f^N(k)$ on the intervals $[0.5,1.0]$, $[3.0,3.5]$,
and $[6.0,6.5]$ in \cref{subfig:barbell_coeffs}.
As expected, the coefficients decay exponentially
to zero, with more terms required at higher
frequencies.

We compute the roots of these Chebyshev interpolants
and apply the post-processing described above.
There were 135 total roots: 3 were removed because
the imaginary part was too large and 13 pairs were found
with values within $\sqrt{\epscheb}$ of each other.
For these 13 pairs, none represented two distinct
eigenvalues or a double root.
This leaves 119 roots in the range $0\leq k \leq 6.5$.
We plot the smallest singular value of
$\cI^N-\cDk^N-2\cW^N$ for each of these roots in
\cref{subfig:barbell_sings}
and plot the vorticity of the eigenfunctions
in \cref{fig:barbell_gallery}.
The singular values suggest that the quality of the
eigenvalues is good.
From the plots, we see that localization occurs
until about the 100th eigenvalue.

%

%\subsection{Robustness on a nearly multiply connected domain}
%\label{subsec:crescent}

\subsection{Eigenvalues of a domain with several inclusions}

\begin{figure}
  \centering
  \includegraphics[width=\textwidth]{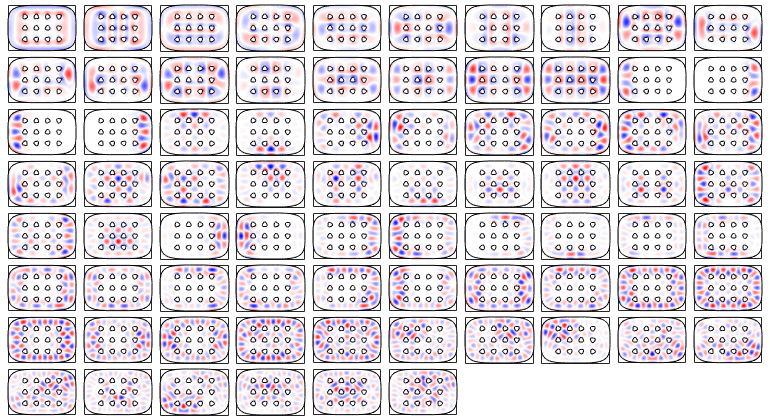}
  \caption{Vorticity plots of the eigenfunctions corresponding
  to the first 76 eigenvalues of a domain with several inclusions.}
  \label{fig:many_inclusions_gallery}
\end{figure}

\begin{figure}
  \centering
  \begin{subfigure}[t]{0.4\textwidth}
    \centering
    \includegraphics[width=\textwidth]{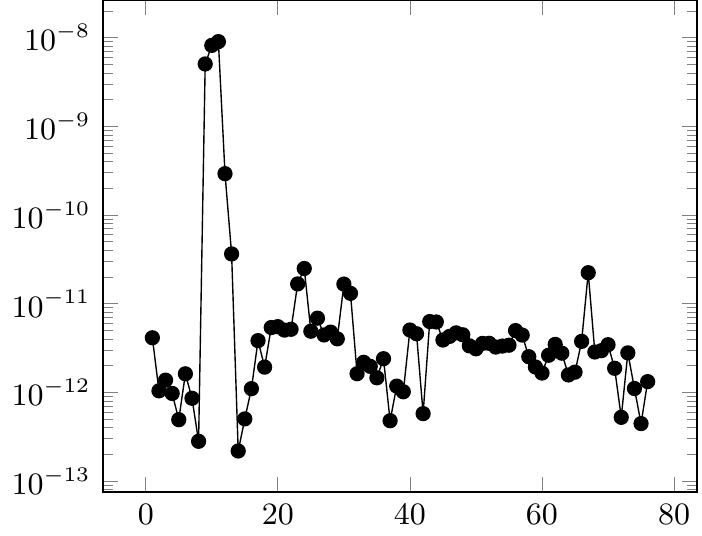}
    \caption{Smallest singular value of the BIE operator
      corresponding to the first 76 eigenvalues of a
      domain with several inclusions.}
    \label{subfig:many_inclusions_sings}
  \end{subfigure}
  ~
  \begin{subfigure}[t]{0.4\textwidth}
    \centering
    \includegraphics[width=\textwidth]{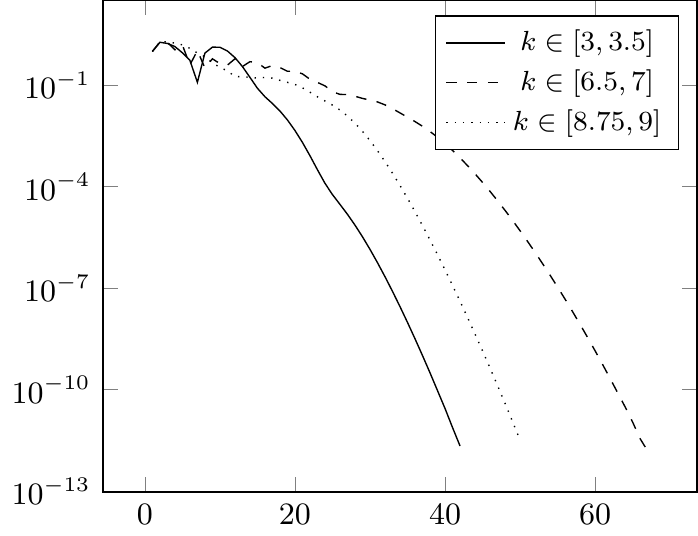}
    \caption{Normalized Chebyshev coefficients of $f^N(k)$ on
      3 different intervals in $k$.}
    \label{subfig:many_inclusions_coeffs}
  \end{subfigure}
  \caption{Diagnostics for the eigenvalues of a domain
    with several inclusions.}
  \label{fig:many_inclusions_diagnostics}
\end{figure}

We now consider the multiply-connected domain in
\cref{subfig:many_inclusions_bdry}.
The domain is defined by a smooth rectangular region
of width 3 and height 2,
with an array of randomly rotated ``starfish'' shapes
removed. 
Such shapes are of interest in
materials design, see, for instance, \cite{overvelde2012compaction}.
We compute the eigenfunctions corresponding to
eigenvalues $k^2$ with $k$ in the range
$3 \leq k \leq 9$ (this range includes the smallest
eigenvalue).

For this smooth shape, ensuring that no panel is
larger than one wavelength for the largest $k$
(here $\lambda=2\pi/9$) is sufficient to resolve
the object to high precision.
After enforcing the
level-restriction property described in
\cref{rmk:levelrestrict}, we end up with
$N_p = 224$.

As this is a multiply-connected domain,
the eigenvalues are estimated by finding the values
$k$ for which $\cI-2\cDk-2i\cSk-2\cW$ is non-invertible.
Let $f^N(k) = \det (\cI^N-2\cDk^N-2i\cSk^N-2\cW^N)$.
To find the roots of $f^N(k)$, we fit a \texttt{chebfun}
representation of $f^N(k)$ on each of the intervals
$[j/2,(j+1)/2]$ for $j = 6,\ldots,13$ and the intervals
$[j/4,(j+1)/4]$ for $j = 28,\ldots,35$.
It should be noted that, due to the relative sizes
of the domains,
this represents a lower frequency problem than
that for the barbell when measured in the number
of wavelengths across the object.
Thus, the use of a finer grid in frequency results from
the difficulty in resolving the Fredholm determinant
for this problem, which has a larger dynamical range
than that for the barbell.
We plot the absolute value of the Chebyshev coefficients
of $f^N(k)$ on the intervals $[3,3.5]$, $[6.5,7]$,
and $[8.75,9]$ in \cref{subfig:barbell_coeffs}.
As expected, the coefficients decay exponentially
to zero, with more terms required at higher
frequencies (note that the interval $[8.75,9]$ is
smaller than the others).

We compute the roots of these Chebyshev interpolants
and apply the post-processing described above.
There were 103 total roots: 21 were removed because
the imaginary part was too large and 6 pairs were found
with values within $\sqrt{\epscheb}$ of each other.
For these 6 pairs, none represented two distinct
eigenvalues or a double root.
This leaves 76 roots in the range $0\leq k \leq 9$.
We plot the smallest singular value of
$\cI^N-\cDk^N-2i\cSk^N-2\cW^N$ for each of these roots in
\cref{subfig:many_inclusions_sings}
and plot the vorticity of the eigenfunctions
in \cref{fig:many_inclusions_gallery}.
In the vorticity plots, we observe a different type of
localization property than that seen in the barbell,
with many of the eigenfunctions
approximately supported in a small, connected subset
of the domain. This is consistent with other studies
\cite{filoche2009strong,lindsay2018boundary}.

\begin{figure}
  \centering
  \begin{subfigure}[t]{0.4\textwidth}
    \centering
    \includegraphics[width=\textwidth]{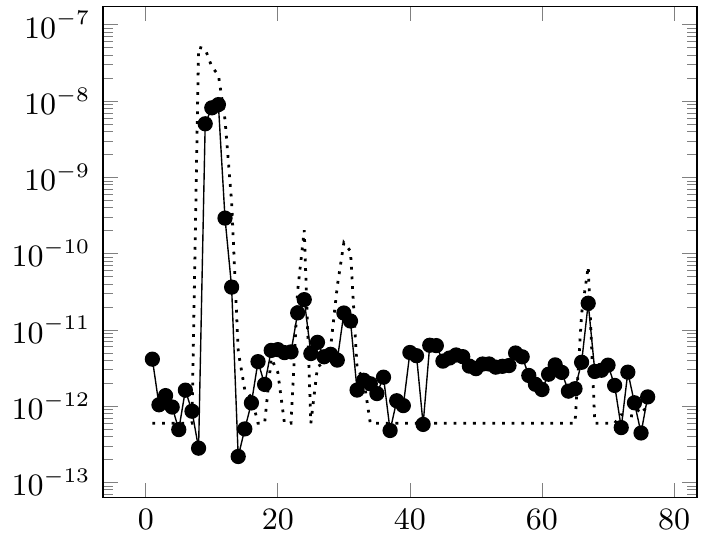}
    \caption{Smallest singular value of the BIE operator
      for the computed roots on the original intervals (\textbullet)
      and an estimate of the error (dotted).}
    \label{subfig:many_inclusions_sings_west}
  \end{subfigure}
  ~
  \begin{subfigure}[t]{0.4\textwidth}
    \centering
    \includegraphics[width=\textwidth]{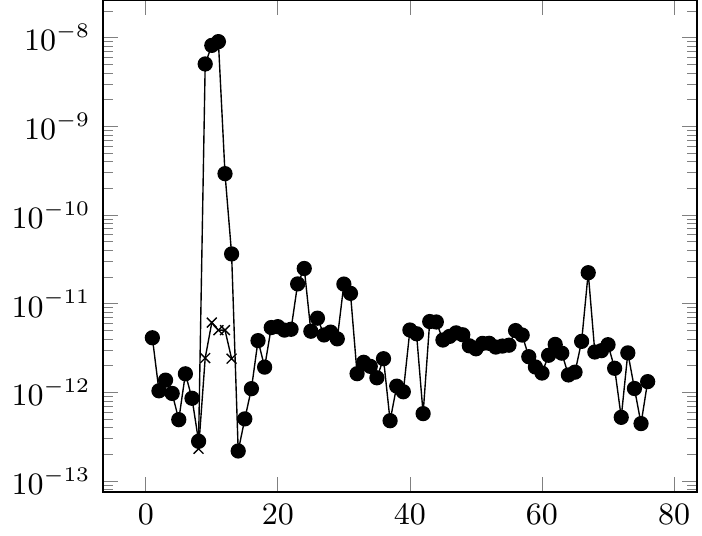}
    \caption{Smallest singular value of the BIE operator
      for the computed roots on the original intervals (\textbullet)
      and the values obtained on the refined interval ($\times$).}
    \label{subfig:many_inclusions_sings_ref}
  \end{subfigure}
  \caption{Further diagnostics for the eigenvalues of a domain
    with several inclusions.}
  \label{fig:many_inclusions_diagnostics_2}
\end{figure}

The singular values suggest that the quality of the
eigenvalues is good, with a few outliers.
To explain these outliers, we consider two quantities
which affect the singular value at a computed root.
As described above, we can approximate the
error in the computed root at $k_\cheb$ by
$\epscheb \|P\|_\infty/|P'(k_\cheb)|$,
where $P$ is the interpolating polynomial.
The singular value estimate itself is affected by
the error incurred in applying the inverse of the
compressed BIE matrix, which can be hard to quantify
\cite{ho2012fast}.
We approximate this error by $O(\sqrt{N})\epsflam$
and assume this is the order of the error in
the singular value estimate.
We plot the maximum of these two estimates
along with the computed singular values in
\cref{subfig:many_inclusions_sings_west}.
There is a reasonably good correlation between
the maximum of the error estimates and the
observed smallest singular value for the BIE,
especially for larger errors.

The worst outliers are from the left half
of the interval $[4.5,5]$.
Because the determinant is much larger on the
right half than the left half of $[4.5,5]$,
we can improve the estimate for the error
in the roots by subdividing the interval. 
We plot the smallest singular value of the
BIE for the roots obtained by fitting a polynomial
on $[4.5,4.75]$ in \cref{subfig:many_inclusions_sings_ref};
the roots on the refined interval
are of significantly higher quality.

\begin{remark}
  The above experience suggests that
  the ratio $\|P\|_\infty/|P'(k_\cheb)|$
  is a useful diagnostic for performing
  automated eigenvalue estimation.
  Note that at a multiple root, this ratio
  will be more difficult to bound.
\end{remark}
%

%

% flatex input end: [04numerical.tex]

%
% flatex input: [05conclusion.tex]
\section{Conclusion}
\label{sec:conclusion}

In the preceding, we have demonstrated a
BIE framework for computing the eigenvalues
of the Stokes operator in the plane which
is robust and scalable.
To justify the
approach, we developed a uniqueness theory
for the oscillatory Stokes equations in
exterior domains in analogy with the
discussion of the Helmholtz equation in
\cite{colton1983integral}.
This lead to the primary theoretical
results of the paper which show that the
BIEs resulting from double layer and
combined-field representations of the
velocity field are 
not invertible precisely when $k^2$ is
an eigenvalue on simply connected
and multiply connected domains, respectively.
As in \cite{zhao2015robust}, the costliness
of performing the nonlinear minimization
associated with computing these eigenvalues
can be alleviated by computing instead the
approximate zeros of the discrete Fredholm
determinant.

The results of this paper can be
extended in a number of ways.
The theory
extends directly to three dimensions, where
computational efficiency and numerical
implementation will be the primary concern.
In the numerical examples above, we
consider only domains with differentiable
boundaries for simplicity.
When using the eigenfunctions as a trial
basis for simulating the Navier--Stokes
equations~\cite{batcho1994generalized},
it is necessary to handle domains
with corners because the domains of interest
arise from domain decomposition, e.g.
dividing a larger domain into 
quadrilaterals.
Fortunately, there has been recent progress
toward efficient discretization of the
layer potentials of elliptic operators
on domains with corners
\cite{helsing2008corner,serkh2016solution,rachh2017solution,helsing2018integral}
which makes the solution of such problems
tractable.
Of course, the theoretical considerations
are different for such domains.
Computing the Stokes eigenvalues of regions
with corners is the topic of ongoing
research.

As noted in the introduction, the
eigenvalues of the Stokes operator are
equivalent to the so-called
``buckling'' eigenvalues of the
biharmonic operator on simply connected
domains~\cite{kelliher2009eigenvalues}.
This can be seen through the stream function
formulation  of the oscillatory Stokes
equation, i.e. setting $\bu = \nabla^\perp \psi$
where $\psi$ now satisfies

\begin{equation*}
  -\Delta^2 \psi = k^2 \Delta \psi \quad \textrm{in} \quad \Omega\; .
\end{equation*}
Note that the buckling problem enforces the
clamped, or first Dirichlet, boundary
condition on $\psi$

\begin{equation*}
  \psi = \partial_\nu \psi = 0 \; \quad \textrm{on} \quad \Gamma.
\end{equation*}
On a multiply connected domain, there are
Stokes eigenfunctions which do not have a corresponding
clamped stream function, so that the
buckling eigenvalues are a subset of the Stokes
eigenvalues.
By adapting the approach of \cite{rachh2017integral},
a suitable layer potential representation of the
buckling problem can be derived based on the
oscillatory Stokes layer potentials.
This is the subject of a follow-up paper which is in
preparation.

There are some interesting questions to answer
on the use of Fredholm determinants in numerical
calculations.
As observed in \cite{zhao2015robust}, the
combined-field representation causes some
difficulty in that the Fredholm determinant
is not defined when the single layer, which
is not trace-class, is included.
Zhao and Barnett~\cite{zhao2015robust}
suggest looking into
representations of the form $\cI-2\cD
-2i\eta\cS^2-2\cW$,
which would have a well-defined Fredholm
determinant.
The relative performance of
such an approach should be explored.
Further, as discussed above, the
convergence of the determinant of
integral equations discretized with
panel-corrected schemes (as described in
\cref{sec:numerical}) is yet to be proved.
When addressing high frequency problems,
the fast-direct solver used to evaluate
determinants in this paper will no longer
have near-linear scaling~\cite{ho2012fast}.
The design of fast-direct solvers in this
regime is the subject of ongoing research, and
to the best of our knowledge, fast determinant
computation at high frequency is an open
question.
Finally, it is worth exploring alternatives
to the Fredholm determinant which perform well
for layer potentials that are not trace class
on the boundary or for problems at higher
frequencies.

% flatex input end: [05conclusion.tex]

%
% flatex input: [06acknowledgements.tex]
\section{Acknowledgements}
The authors would like to thank Alex Barnett, Leslie Greengard, and Shidong Jiang
for many useful discussions.

T. Askham was supported by the Air Force Office of Scientific Research 
under Grant FA9550-17-1-0329.

% flatex input end: [06acknowledgements.tex]

%

\appendix

% flatex input: [appendix.tex]
\section{Dirichlet eigenvalues and eigenfunctions on the annulus \label{sec:annul_dir_exact}}
In this section, we compute some of the Dirichlet eigenvalues
corresponding to a subset of the radially symmetric eigenfunctions
on the annulus. 
In polar coordinates $(r,\theta)$, consider the annulus defined by 
$R_{1}<r<R_{2}$. 
Suppose that $\bu$ is of the form 
\begin{equation}
\bu = \nabla^{\perp}\left( \alpha H_{0}(kr) + \beta J_{0}(kr) \right) \, ,
\label{eq:annulus_direig}
\end{equation}
and $p=0$.

Clearly, this pair satisfies the osciallatory Stokes equation
with parameter $k$, since $J_{0}(kr)$ and $H_{0}(kr)$ 
satisfy the Helmholtz equation on the annulus.

Let $\hat{r},\hat{\theta}$ denote the unit vectors in polar coordinates.
A simple calculation shows that 
\begin{equation}
\begin{aligned}
u_{r} &= \bu \cdot \hat{r}  = 0 \, \\
u_{\theta} &= \bu \cdot \hat{\theta} = k(\alpha  H_{0}'(kr) + \beta J_{0}'(kr)) \, .
\end{aligned}
\end{equation}

This in particular implies that on $r=R_{1}$,
$u_{\theta}$ takes on the constant value,
$k(\alpha H_{0}'(kR_{1}) + \beta J_{0}'(kR_{1}))$.
Similarly, on $r=R_{2}$, 
$u_{\theta}$ takes on the constant value,
$k(\alpha H_{0}'(kR_{2}) + \beta J_{0}'(kR_{2}))$.

Thus, if $k$ satisfies, 
\begin{equation}
H_{0}'(kR_{1}) J_{0}'(kR_{2}) - H_{0}'(kR_{2})J_{0}'(kR_{1}) = 0 \, ,
\end{equation}
and for those values of $k$ if 
$\alpha,\beta$ are non-zero solutions to system of equations 
\begin{equation}
\begin{bmatrix}
H_{0}'(kR_1) & J_{0}'(kR_{1}) \\
H_{0}'(kR_{2}) & J_{0}'(kR_{2}) 
\end{bmatrix}
\begin{bmatrix}
\alpha \\
\beta
\end{bmatrix}
=
\begin{bmatrix}
0 \\
0
\end{bmatrix}
\, ,
\end{equation}
then $k$ is a Dirichlet eigenvalue and $\bu$ 
defined by~\cref{eq:annulus_direig}
is the corresponding eigenfunction.

\section{Neumann eigenvalues and eigenfunctions on the unit disk}
In this section, we derive an analytical expression which
can be used to compute some of the radially symmetric 
Neumann eigenvalues on the unit disk for the Stokes operator.

Suppose that $\bu$ is of the form
\begin{equation}
\bu = \nabla^{\perp} \jkr \, ,
\end{equation}
and the pressure is given by $p=0$, as
$\bu$ satisfies $(\Delta + k^2)\bu = 0$.

Then, the surface traction $\bt$ on the disk of radius $r$ 
is given by
\begin{equation}
\bt = 
\left( -\frac{k}{r^2}\jpkr  + \frac{k^2}{r} \jppkr\right)
\begin{bmatrix}
\sint \\
-\cost
\end{bmatrix} \, .
\end{equation}

Thus, $k$ which satisfies
\begin{equation}
  \label{eq:neu_roots}
-k \jpk + k^2 \jppk = 0 \, ,
\end{equation}
are Neumann eigenvalues on the unit disk.
The first $4$ roots of the \cref{eq:neu_roots} are
in \cref{tab:neu_roots}.
\begin{table}
  \centering
  \begin{tabular}{c}
    k  \\ \hline
    5.135622301840683 \\
    8.417244140399865 \\
    11.61984117214906 \\
    14.79595178235126 
  \end{tabular}
  \caption{Roots of \cref{eq:neu_roots}.}
  \label{tab:neu_roots}
\end{table}

% flatex input end: [appendix.tex]

%
%*flatex input: [stokes-eig.bbl]

% flatex input end: [stokes-eig.bbl]
%FLATEX-REM:\bibliographystyle{plain}
%FLATEX-REM:\bibliography{refs}

\end{document}